\pdfoutput=1

\documentclass[nospthms]{svjour3} 
\usepackage{mathptmx}

\smartqed

\usepackage[utf8]{inputenc}

\usepackage{amssymb,amstext,amsmath,amsthm}
\usepackage{bm}

\usepackage{graphicx,wrapfig}
\usepackage[hang]{subfigure}
\usepackage[all]{xy}

\usepackage{url}

\usepackage{hyperref}

\usepackage[numbers,sort&compress]{natbib}

\usepackage{pdfsync}

\usepackage{theomac}

\usepackage{theomac}

\newcommand{\lemlab}[1]{\label{lemma:#1}}
\newcommand{\theolab}[1]{\label{theo:#1}}

\newcommand{\eqlab}[1]{\label{eq:#1}}
\newcommand{\corlab}[1]{\label{cor:#1}}

\newcommand{\seclab}[1]{\label{section:#1}}

\newcommand{\lemref}[1]{Lemma \ref{lemma:#1}}

\newcommand{\theoref}[1]{Theorem \ref{theo:#1}}

\newcommand{\corref}[1]{Corollary \ref{cor:#1}}

\newcommand{\figref}[1]{Figure \ref{fig:#1}}
\renewcommand{\eqref}[1]{(\ref{eq:#1})}
\newcommand{\secref}[1]{Section \ref{section:#1}}

\newtheoremWithMacro{theorem}{Theorem}

\newtheorem{lemma}{Lemma}

\newtheorem{cor}[lemma]{Corollary}

\def\lemmaD#1{
\begin{lemma}
\label{lem:#1}
}

\newcommand{\theoremD}[1]{
\begin{theorem}
\label{theorem:#1}
}

\newcommand{\factD}[1]{
\begin{fact}
\label{fact:#1}
}

\newcommand{\corD}[1]{
\begin{cor}
\label{cor:#1}
}

\newcommand{\card}[1]{\ensuremath{\left\vert #1 \right\vert}}
\renewcommand{\vec}[1]{\ensuremath{\mathbf{#1}}}

\newcommand{\iprod}[2]{\left\langle {#1},{#2}\right\rangle}
\renewcommand{\Re}{\mathbb{R}}

\makeatletter
\def\paragraph{\@startsection{paragraph}{4}{\z@}{-13pt plus-8pt minus-4pt}{\z@}{\normalsize\bf}}
\def\subsection{\@startsection{subsection}{2}{\z@}%
{-21dd plus-8pt minus-4pt}{10.5dd}
{\normalsize\upshape\bf}}
\def\subsubsection{\@startsection{subsubsection}{3}{\z@}%
{-13dd plus-8pt minus-4pt}{10.5dd}
{\normalsize\upshape\bf}}
\makeatother

\journalname{Discrete and Computational Geometry}
\title{Slider-pinning Rigidity: a Maxwell-Laman-type Theorem}

\author{Ileana Streinu \and Louis Theran}

\institute{Ileana Streinu \at
Computer Science Department, Smith College, Northampton, MA \\
\email{istreinu@smith.edu,streinu@cs.smith.edu} \and
Louis Theran \at
Mathematics Department, Temple University, Philadelphia, PA \\
\email{theran@temple.edu}}

\date{Received: date / Accepted: date}

\begin{document}
\maketitle

\begin{abstract}
We  define and study slider-pinning rigidity,
giving a complete combinatorial characterization.  This is
done via direction-slider networks, which are a generalization of
Whiteley's direction networks.
\end{abstract}

\section{Introduction}
A planar \emph{bar-and-joint framework} is a planar structure made of fixed-length bars connected by
universal joints with full rotational degrees of freedom.
The allowed continuous motions preserve the lengths and connectivity of the bars.  Formally, a bar-and-joint
framework is modeled as a pair $(G,\bm{\ell})$, where $G=(V,E)$ is a simple graph with $n$ vertices and $m$
edges, and $\bm{\ell}$ is a vector of positive numbers that are interpreted as squared edge lengths.

A realization $G(\vec p)$ of a bar-and-joint framework is a mapping of the vertices of $G$ onto
a point set $\vec p\in (\Re^2)^n$ such that $||\vec p_i-\vec p_j||^2=\bm{\ell}_{ij}$ for every
edge $ij\in E$.  The realized framework $G(\vec p)$ is rigid if the only motions are trivial
rigid motions; equivalently, $\vec p$ is an isolated (real) solution to the equations
giving the edge lengths, modulo rigid motions.  A framework $G(\vec p)$ is \emph{minimally rigid} if it is rigid, but ceases to be so if any bar is removed.

\paragraph{The Slider-pinning Problem.}
In this paper, we introduce an elaboration of planar bar-joint rigidity to include
\emph{sliders}, which constrain some of the vertices of a framework to move on  given lines.  We define the combinatorial model for a bar-slider framework to be a graph $G=(V,E)$ that has edges (to represent the bars) and also self-loops (that represent the sliders).

A realization of a bar-slider framework $G(\vec p)$ is a mapping of the vertices of $G$ onto a point set that is compatible with the given edge lengths, with the additional requirement that if a vertex is on a slider, then it is mapped to a point on the slider's line.  A bar-slider framework $G(\vec p)$ is \emph{slider-pinning rigid} (shortly \emph{pinned}) if it is completely immobilized.  It is minimally pinned if it is pinned and ceases to be so when any bar or slider is removed.  (Full definitions are given in \secref{rigidity}).

\paragraph{Historical note on pinning frameworks.}
The topic of immobilizing bar-joint frameworks has been considered before.
Lovász \cite{lovas:1980:matroidmatching} and, more recently, Fekete \cite{fekete:thumbtacking}
studied the related problem of pinning a bar-joint frameworks by a minimum number of
\emph{thumbtacks}, which completely immobilize a vertex.  Thumbtack-pinning
induces a different (and non-matroidal) graph-theoretic structure than slider-pinning.
In terms of slider-pinning, the minimum thumbtack-pinning problem asks for a slider-pinning
with sliders on the minimum number of distinct vertices.  Recski \cite{recski:matroidTheory:1989a}
also previously considered the specific case of vertical sliders, which he called tracks.

We give, for the first time, a
\emph{complete combinatorial characterization} of planar slider-pinning in the most general setting.
Previous work on the problem is concerned either with thumbtacks (Fekete \cite{fekete:thumbtacking})
or only with the algebraic setting (Lovász \cite{lovas:1980:matroidmatching}, Recski \cite{recski:matroidTheory:1989a}).

On the algorithmic side, we \cite{sliders} have previously developed algorithms for generic rigidity-theoretic questions on bar-slider frameworks.  The theory developed in this paper provides the theoretical foundation for their correctness.

\paragraph{Generic combinatorial rigidity.}
The purely geometric question of deciding rigidity of a framework seems to be computationally intractable, even for small, fixed dimension $d$.  The best-known algorithms rely on exponential time Gröbner basis techniques, and specific cases are known to be NP-complete \cite{saxe:embeddability:1979}.  However, for \emph{generic} frameworks in the plane, the following landmark theorem due to Maxwell and Laman states that rigidity has a combinatorial characterization, for which several efficient algorithms are known (see \cite{pebblegame} for a discussion of the algorithmic aspects of rigidity).  The Laman graphs and looped-Laman graphs appearing in the statements of results are combinatorial (not geometric) graphs with special sparsity properties. The technical definitions are given in \secref{sparse}.

\begin{theorem}[\laman][\textbf{Maxwell-Laman Theorem: Generic bar-joint rigidity \cite{laman,MaxwellEquil1864}}]\theolab{laman}
Let $(G,\bm \ell)$ be a generic abstract bar-joint framework.  Then $(G,\bm \ell)$
is minimally rigid if and only if $G$ is a Laman graph.
\end{theorem}

Our main rigidity result is a Maxwell-Laman-type theorem for slider-pinning rigidity.

\begin{theorem}[\slider][\textbf{Generic bar-slider rigidty}]\theolab{slider}
Let $(G,\vec \ell,\vec n,\vec s)$ be a generic bar-slider framework.  Then
$(G,\vec \ell,\vec n,\vec s)$ is minimally rigid if and only if $G$ is looped-Laman.
\end{theorem}

Our proof relies on a new technique and proceeds via direction networks, defined next.

\paragraph{Direction networks.}
A \emph{direction network} $(G,\vec d)$ is a graph $G$ together with an
assignment of a direction vector $\vec d_{ij}\in \Re^2$ to each edge.  A
realization $G(\vec p)$ of a direction network is an embedding of $G$
onto a point set $\vec p$ such that $\vec p_i-\vec p_j$ is in the direction
$\vec d_{ij}$; if the endpoints of every edge are distinct, the realization is
\emph{faithful}.

The \emph{direction network realizability problem} is to find a realization
$G(\vec p)$ of a direction network $(G,\vec d)$.

\paragraph{Direction-slider networks.}
We define a \emph{direction-slider network} $(G,\vec d,\vec n,\vec s)$ to be
an extension of the direction network model to include sliders.  As in slider-pinning rigidity, the
combinatorial model for a slider is defined to be a self-loop in the graph $G$.
A realization $G(\vec p)$ of a direction-slider network respects the given direction
for each edge, and puts $\vec p_i$ on the line specified for each slider.  A realization is
\emph{faithful} is the endpoints of every edge are distinct.

\paragraph{Generic direction network realizability.}
Both the direction network realization problem and the direction-slider network
realization problem give rise to a \emph{linear} system of equations, in contrast
to the quadratic systems arising in rigidity, greatly simplifying the analysis
of the solution space.

The following theorem was proven by Whiteley.

We give a new proof, using different geometric and combinatorial techniques, and we give
an explicit description of the set of generic directions.

\begin{theorem}[\parallelthm][\textbf{Generic direction network realization
(Whiteley \cite{Whiteley:1989p992,Whiteley:1988p137,whiteley:Matroids:1996})}]\theolab{parallel}
Let $(G,\vec d)$ be a generic direction network, and let $G$ have $n$
vertices and $2n-3$ edges.
Then $(G,\vec d)$ has a (unique, up to translation and rescaling) faithful realization if and only if $G$ is a Laman graph.
\end{theorem}

For direction-slider networks we have a similar result to \theoref{parallel}.

\begin{theorem}[\sliderparallelthm][\textbf{Generic direction-slider network realization}]\theolab{sliderparallel}
Let $(G,\vec d,\vec n,\vec s)$ be a generic direction-slider network.  Then
$(G,\vec d,\vec n,\vec s)$ has a (unique) faithful realization if and only if
$G$ is a looped-Laman graph.
\end{theorem}

\paragraph{From generic realizability to generic rigidity.}
Let us briefly sketch how the rigidity theorems \ref{theo:laman} and
\ref{theo:slider} follow from the direction network realization theorems
\ref{theo:parallel} and \ref{theo:sliderparallel} (full details are given in
\secref{rigidity}).
For brevity, we sketch only how \theoref{parallel} implies \theoref{laman}, an implication that can be traced back to Whiteley in \cite{Whiteley:1989p992}.
The proof that \theoref{sliderparallel} implies \theoref{slider} will follow a similar proof plan.

All known proofs of the Maxwell-Laman theorem proceed
via \emph{infinitesimal rigidity}, which is a linearization of the rigidity
problem obtained by taking the differential of the system of equations
specifying the edge lengths 
and sliders 
to obtain the \emph{rigidity matrix} $\vec M_{2,3}(G)$
of the abstract framework (see \figref{rigidity-matrices}(a)).

One then proves the following two statements about bar-joint frameworks
$(G,\bm{\ell})$ with $n$ vertices and $m=2n-3$ edges:
\begin{itemize}
\item In realizations where the rigidity matrix achieves rank $2n-3$
the framework is rigid.
\item The rigidity matrix achieves rank $2n-3$
for almost all realizations (these are called \emph{generic}) if and only if the graph $G$
is Laman.
\end{itemize}

The second step, where the rank of the rigidity matrix is established from only a combinatorial assumption, is the (more difficult) ``Laman direction''.  The plan is in two steps:
\begin{itemize}
\item We begin with a matrix $\vec M_{2,2}(G)$, arising from the direction network realization problem,
that has non-zero entries in the same positions as the rigidity matrix, but a simpler
pattern: $\vec d_{ij}=(a_{ij},b_{ij})$ instead of $\vec p_i-\vec p_j$ (see \figref{m22}).
The rank of the simplified matrices is established in \secref{natural} via a
matroid argument.
\item We then apply the direction network realization  \theoref{parallel}
to a Laman graph.  For generic (defined in detail in \secref{parallel})
edge directions $\vec d$ there exists a point set $\vec p$
such that $\vec p_i-\vec p_j$ is in the direction $\vec d_{ij}$, with $\vec p_i\neq \vec p_j$
when $ij$ is an edge.  Substituting the $\vec p_i$ into $\vec M_{2,2}(G)$ recovers the
rigidity matrix while preserving rank, which completes the proof.
\end{itemize}

\paragraph{Genericity.}
In this paper, the term \emph{generic} is used in the standard sense of algebraic geometry: a property is generic
if it holds on the (open, dense) complement of an algebraic set defined by a finite number of polynomials.
In contrast, the rigidity literature employs a number of definitions that are not as amenable to combinatorial or
computational descriptions. Some authors \cite[p. 92]{lovasz:yemini} define a {\em generic framework} as being one where the points $\vec p$ are algebraically independent.  Other frequent definitions used in rigidity theory require that generic properties hold {\em for most of} the point sets (measure-theoretical) \cite[p. 1331]{whiteley:surveyHandbook:2004} or focus on properties which, if they hold for a point set $\vec p$ (called generic for the property), then they hold for any point in some open neighborhood (topological) \cite{gluck:almostAll:1975}.

For the specific case of Laman bar-joint rigidity we identify two types of conditions on the defining polynomials:
some arising from the genericity of directions in the direction network with the same graph as the framework being analyzed;
and a second type arising from the constraint the the directions be realizable as the difference set of a planar point set.
To the best of our knowledge, these observations are new.

\paragraph{Organization.} The rest of this paper is organized as follows.
\secref{sparse} defines Laman and looped-Laman graphs and gives the combinatorial tools from the theory of $(k,\ell)$-sparse and
$(k,\bm{\ell})$-graded sparse graphs that we use to analyze direction networks and
direction-slider networks.  \secref{natural} introduces the needed results about
$(k,\ell)$-sparsity-matroids, and we prove two matroid representability results for
the specific cases appearing in this paper.  \secref{parallel} defines
direction networks, the realization problem for them, and proves \theoref{parallel}.
\secref{sliderparallel} defines slider-direction networks and proves the analogous
\theoref{sliderparallel}. In \secref{xyparallelslider} we extend \theoref{sliderparallel}
to the specialized situation where all the sliders are axis-parallel.

In \secref{rigidity} we move to the setting of frameworks, defining bar-slider rigidity and
proving the rigidity Theorems \ref{theo:laman} and \ref{theo:slider} from our results on
direction networks.  In addition, we discuss the relationship between our work and
previous proofs of the Maxwell-Laman theorem.

\paragraph{Notations.}
Throughout this paper we will use the notation  $\vec p \in (\Re^2)^n$ for a set of $n$ points in the plane.  By identification of $(\Re^2)^n$ with $\mathbb{R}^{2n}$, we can think of $\vec p$ either as a vector of point $\vec p_i=(a_i,b_i)$ or as a flattened vector $\vec p=(a_1,b_1,a_2,b_2,\ldots,a_n,b_n)$.
When points are used as unknown variables, we denote them as $\vec p_i=(x_i,y_i)$.

Analogously, we use the notation $\vec d\in (\Re^2)^m$ for a set of $m$ directions in $\Re^2$.  Since
directions will be assigned to edges of a graph, we index the entries of $\vec d$ as $\vec d_{ij}=(a_{ij},b_{ij})$ for the direction of the edge $ij$.

The graphs appearing in this paper have edges and also self-loops (shortly, loops).
Both multiple edges and multiple self loops will appear, but the multiplicity will never
be more than two copies. We will use $n$ for the
number of vertices, $m$ for the number of edges, and $c$ for the numbers of self-loops.  Thus
for a graph $G=(V,E)$ we have $\card{V}=n$ and $\card{E}=m+c$.  Edges are written as $(ij)_k$
for the $k$th copy of the edge $ij$, ($k=1,2$).  As we will not usually need to distinguish between copies,
we abuse notation and simply write $ij$, with the understanding that multiple edges are considered
separately in ``for all'' statements.  The $j$th loop on vertex $i$ is denoted $i_j$ ($j=1,2$).

For subgraphs $G'$ of a graph $G$, we will typically use $n'$ for the number of vertices, $m'$
for the number of edge and $c'$ for the number of loops.

A contraction of a graph $G$ over the edge $ij$ (see \secref{sparse} for a complete definition) is
denoted $G/ij$.

We use the notation $[n]$ for the set $\{1,2,\ldots,n\}$. If $\vec A$ is an $m\times n$ matrix, then $\vec A[M,N]$ is the sub-matrix induced by the rows $M\subset [m]$ and $N\subset [n]$.

\section{Sparse and graded-sparse graphs}\seclab{sparse}
Let $G$ be a graph on $n$ vertices, possibly with multiple edges and loops.
$G$ is \emph{$(k,\ell)$-sparse} if for all
subgraphs $G'$ of $G$ on $n'$ vertices, the numbers of induced edges and loops $m'+c'\le kn'-\ell$.
If, in addition, $G$ has $m+c=kn-\ell$ edges and loops, then $G$ is \emph{$(k,\ell)$-tight}.
An induced subgraph of a $(k,\ell)$-sparse graph $G$ that is $(k,\ell)$-tight is called a \emph{block} in $G$; a maximal block is called a \emph{component} of $G$.

Throughout this paper, we will be interested in two particular cases of sparse graphs: $(2,2)$-tight graphs and $(2,3)$-tight graphs.  For brevity of notation we call these \emph{$(2,2)$-graphs} and \emph{Laman graphs} respectively.  We observe that the sparsity parameters of both $(2,2)$-graphs and Laman graphs do not have self-loops.  Additionally, Laman graphs are simple, but $(2,2)$-graphs may have two parallel edges (any more would violate the sparsity condition).  See \figref{22-examples} and \figref{23-examples} for examples.
\begin{figure}[htbp]
\centering

\subfigure[]{\includegraphics[width=0.45\textwidth]{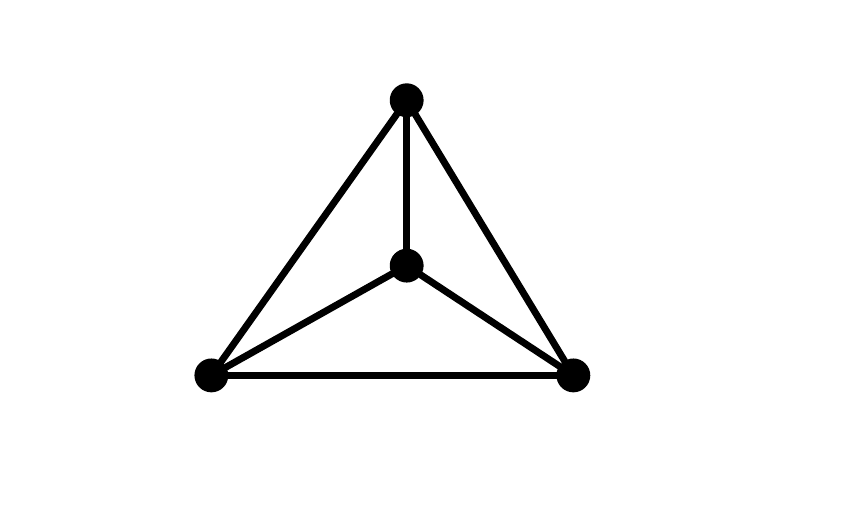}}
\subfigure[]{\includegraphics[width=0.45\textwidth]{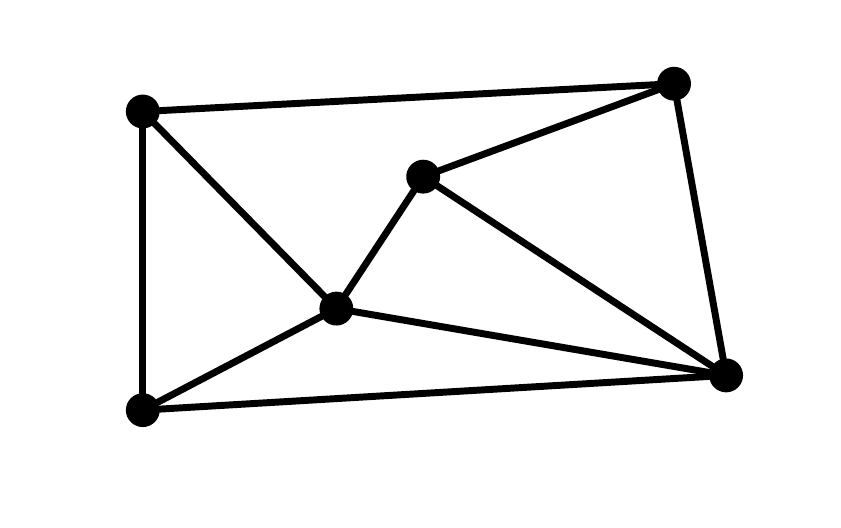}}
\caption{Examples of $(2,2)$-graphs: (a) $K_4$; (b) a larger example on $6$ vertices.}
\label{fig:22-examples}
\end{figure}

\begin{figure}[htbp]
\centering

\subfigure[]{\includegraphics[width=0.45\textwidth]{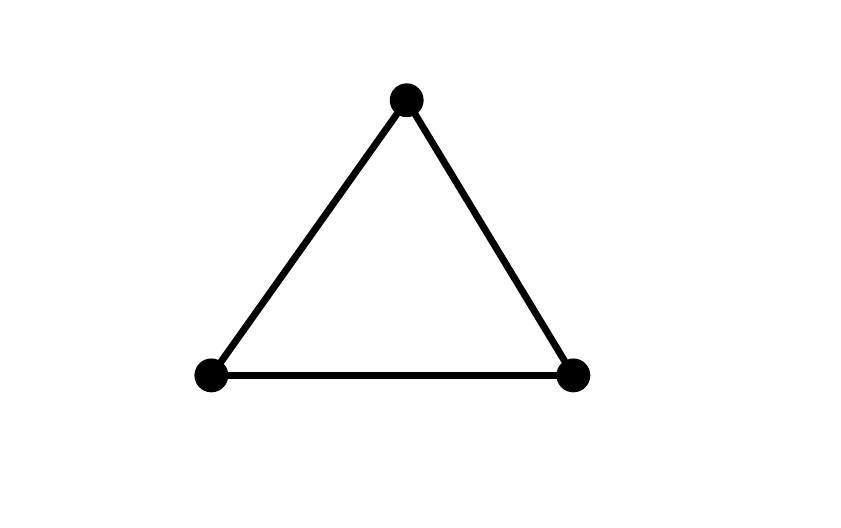}}
\subfigure[]{\includegraphics[width=0.45\textwidth]{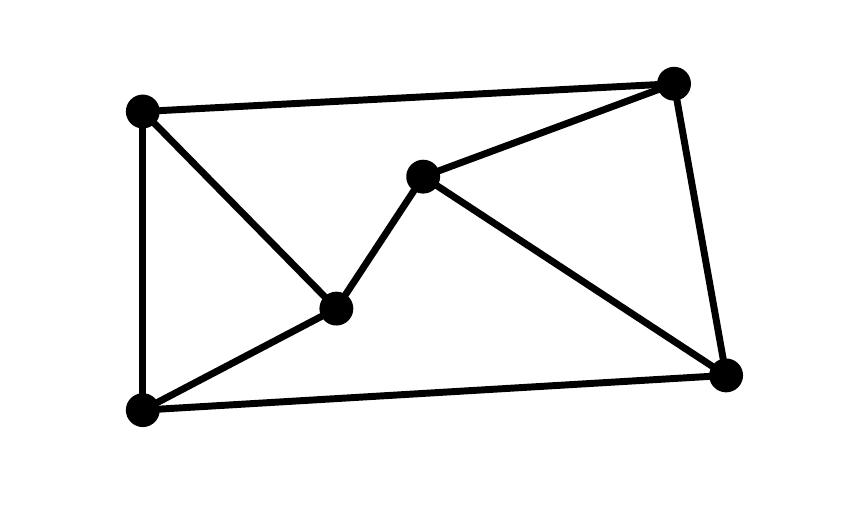}}
\caption{Examples of Laman graphs.}
\label{fig:23-examples}
\end{figure}

\paragraph{Graded sparsity.} We also make use of a specialization of the $(k,\bm{\ell})$-graded-sparse
graph concept from our paper \cite{graded}.  Here, $\bm{\ell}$ is a vector of integers, rather than just a single integer value.
To avoid introducing overly general notation that is immediately specialized, we define it only for the specific parameters we use in this paper.

Let $G$ be a graph on $n$ vertices with edges and also self-loops.  $G$ is $(2,0,2)$-graded-sparse if:
\begin{itemize}
\item All subgraphs of $G$ with only edges (and no self-loops) are $(2,2)$-sparse.
\item All subgraphs of $G$ with edges and self-loops are $(2,0)$-sparse.
\end{itemize}
If, additionally, $G$ has $m+c=2n$ edges and loops, then $G$ is \emph{$(2,0,2)$-tight} (shortly looped-$(2,2)$).  See \figref{looped22-examples} for examples of looped-$(2,2)$ graphs.

\begin{figure}[htbp]
\centering

\subfigure[]{\includegraphics[width=0.45\textwidth]{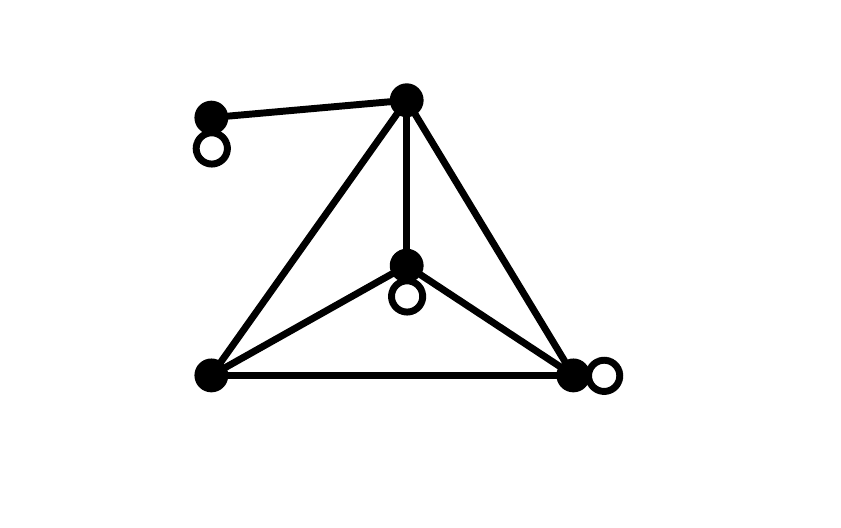}}
\subfigure[]{\includegraphics[width=0.45\textwidth]{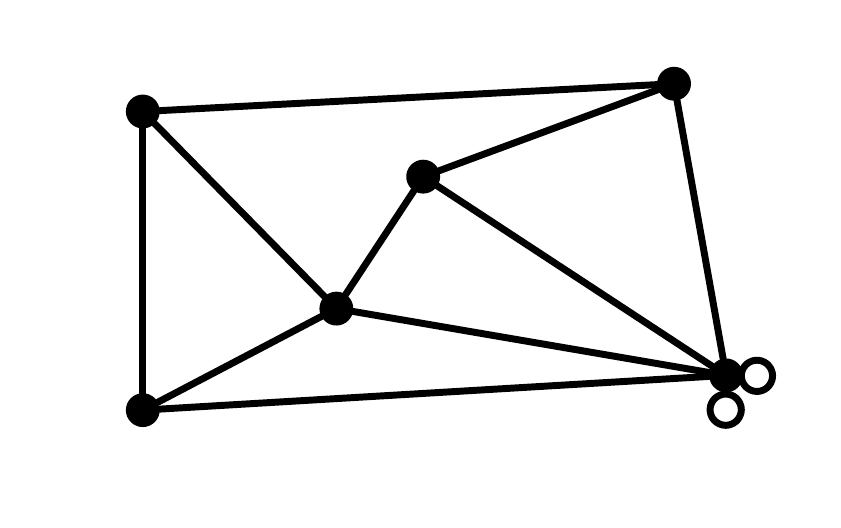}}
\caption{Examples of looped-$(2,2)$ graphs.}
\label{fig:looped22-examples}
\end{figure}

Let $G$ be a graph on $n$ vertices with edges and also self-loops.  $G$ is $(2,0,3)$-graded-sparse if:
\begin{itemize}
\item All subgraphs of $G$ with only edges (and no self-loops) are $(2,3)$-sparse.
\item All subgraphs of $G$ with edges and self-loops are $(2,0)$-sparse.
\end{itemize}
If, additionally, $G$ has $m+c=2n$ edges and loops, then $G$ is \emph{$(2,0,3)$-tight} (shortly looped-Laman).  See \figref{looped23-examples} for examples of looped-Laman graphs.
\begin{figure}[htbp]
\centering

\subfigure[]{\includegraphics[width=0.45\textwidth]{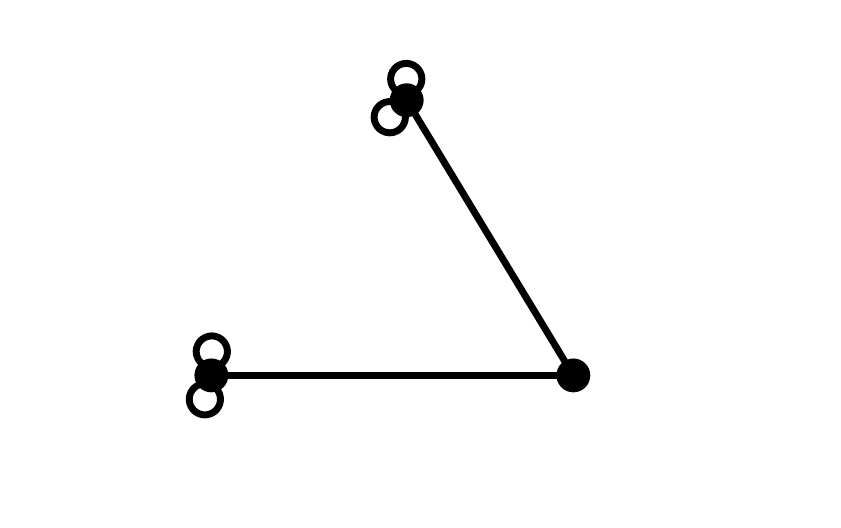}}
\subfigure[]{\includegraphics[width=0.45\textwidth]{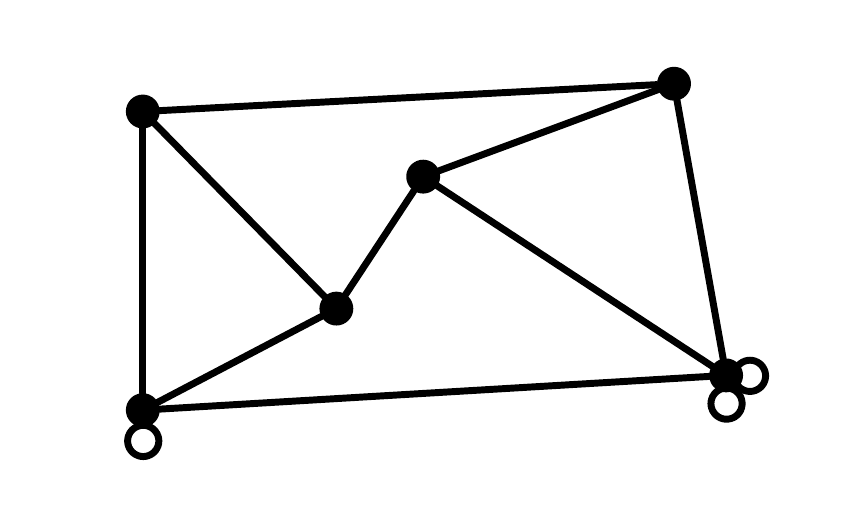}}
\caption{Examples of looped-Laman graphs.}
\label{fig:looped23-examples}
\end{figure}

\paragraph{Characterizations by contractions.}
We now present characterizations of Laman graphs and looped-Laman graphs in terms of
graph contractions.  Let $G$ be a graph (possibly with loops and multiple edges), and let $ij$
be an edge in $G$.  The \emph{contraction of $G$ over $ij$}, $G/ij$ is the graph obtained by:
\begin{itemize}
\item Discarding vertex $j$.
\item Replacing each edge $jk$ with an edge $ik$, for $k\neq i$.
\item Replacing each loop $j_k$ with a loop $i_k$.
\end{itemize}
By symmetry, we may exchange the roles of $i$ and $j$ in this definition without changing it.  We note that
this definition of contraction \emph{retains} multiple \emph{edges} created during the contraction, but that \emph{loops} created by contracting are discarded.  In particular, any loop in $G/ij$
corresponds to a loop in $G$.  \figref{contraction-examples} shows an example of contraction.

\begin{figure}[htbp]
\centering

\subfigure[]{\includegraphics[width=0.45\textwidth]{k3}}
\subfigure[]{\includegraphics[width=0.45\textwidth]{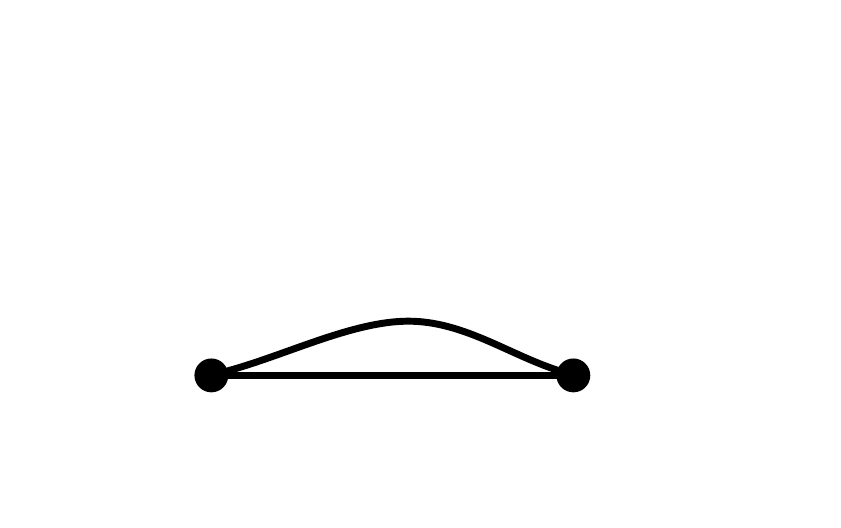}}
\caption{Contracting an edge of the triangle: (a) before contraction; (b) after contraction
we get a doubled edge but \emph{not} a loop, since there wasn't one in the triangle before
contracting.}
\label{fig:contraction-examples}
\end{figure}

The following lemma gives a characterization of Laman graphs in terms of contraction and $(2,2)$-sparsity.
\begin{lemma}\lemlab{22contract}
Let $G$ be a simple $(2,2)$-sparse graph with $n$ vertices and $2n-3$ edges.  Then $G$ is a Laman graph
if and only if after contracting \emph{any} edge $ij\in E$, $G/ij$ is a $(2,2)$-graph on $n-1$
vertices.
\end{lemma}
\begin{proof}
If $G$ is not a Laman graph, then some subset $V'\subset V$ of $n'$ vertices induces a subgraph $G'=(V',E')$ with $m'\ge 2n'-2$ edges.  Contracting any edge $ij$ of $G'$ leads to a contracted graph $G'/ij$ with $n'-1$ vertices and at least $2n'-3=2(n'-1)-1$ edges, so $G'/ij$ is not $(2,2)$-sparse.  Since $G'/ij$ is an induced subgraph of $G/ij$ for this choice of $ij$, $G/ij$ is not a $(2,2)$-graph.

For the other direction, we suppose that $G$ is a Laman graph and fix a subgraph $G'=(V',E')$ induced by
$n'$ vertices.  Since $G$ is Laman, $G'$ spans at most $2n'-3$ edges, and so for any edge $ij\in E'$ the
contracted graph $G'/ij$ spans at most $2n'-4=2(n'-1)-2$ edges, so $G'/ij$ is $(2,2)$-sparse.  Since this
$G'/ij$ is an induced subgraph of $G/ij$, and this argument holds for \emph{any} $V'\subset V$ and edge $ij$, $G/ij$ is $(2,2)$-sparse for any edge $ij$.  Since $G/ij$ has $2n-2$ edges, it must be a $(2,2)$-graph.
\end{proof}

For looped-Laman graphs, we prove a similar characterization.
\begin{lemma}\lemlab{loopedlamancontract}
Let $G$ be a looped-$(2,2)$ graph.  Then $G$ is looped-Laman if and only
if for \emph{any} edge $ij\in E$ there is a loop $v_w$ (depending on $ij$)
such that $G/ij-v_w$ is a looped-$(2,2)$ graph.
\end{lemma}
\begin{proof}
Let $G$ have $n$ vertices, $m$ edges, and $c$ loops.  Since $G$ is looped-$(2,2)$, $2n=m+c$.
If $G$ is not looped-Laman, then by \lemref{22contract}, the edges of $G/ij$ are not $(2,2)$-sparse,
which implies that $G/ij-v_w$ cannot be $(2,0,2)$-graded-sparse for any loop $v_w$ because the loops
play no role in the $(2,2)$-sparsity condition for the edges.

If $G$ is looped-Laman, then the edges will be $(2,2)$-sparse in any contraction $G/ij$.  However, $G/ij$
has $n-1$ vertices, $m-1$ edges and $c$ loops, which implies that $m-1+c=2n-1=2(n-1)+1$, so $G/ij$
is not $(2,0)$-sparse as a looped graph.  We have to show that there is \emph{one} loop, which when removed, restores $(2,0)$-sparsity.

For a contradiction, we suppose the contrary: for any contraction $G/ij$, there is some subgraph $(G/ij)'=(V',E')$
of $G/ij$ on $n'$ vertices inducing $m'$ edges and $c'$ loops with $m'+c'\ge 2n'+2$.  As noted above $m'\le 2n'-2$.  If $(G/ij)'$
does not contain $i$, the surviving endpoint of the contracted edge $ij$, then $G$ was not looped-$(2,2)$,
which is a contradiction.  Otherwise, we consider the subgraph induced by $V'\cup \{i\}$ in $G$.  By
construction it has $n'+1$ vertices, $m'+1$ edges and $c'$ loops.  But then we have $m'+1+c'\ge 2n'+3=2(n'+1)+1$, contradicting $(2,0,2)$-graded-sparsity of $G$.
\end{proof}

\section{Natural realizations for $(2,2)$-tight and $(2,0,2)$-tight graphs}\seclab{natural}
Both $(k,\ell)$-sparse and $(k,\ell)$-graded-sparse graphs form matroids, with the
$(k,\ell)$-tight and $(k,\bm{\ell})$-graded-tight graphs as the bases, which
we define to be the $(k,\ell)$-sparsity-matroid and the $(k,\bm{\ell})$-graded-sparsity matroid,
respectively.   Specialized to our case, we talk about the $(2,2)$- and $(2,3)$-sparsity matroids and the  $(2,0,2)$- and $(2,0,3)$-graded-sparsity matroids, respectively.

In matroidal terms, the rigidity Theorems \ref{theo:laman} and \ref{theo:slider} state that
the rigidity matrices for bar-joint and bar-slider frameworks are \emph{representations} of the
$(2,3)$-sparsity matroid and $(2,0,3)$-graded-sparsity matroid, respectives: linear independence among the
rows of the matrix corresponds bijectively
to independence in the associated combinatorial matroid for generic
frameworks.  The difficulty in the proof is that the pattern of the rigidity matrices $\vec{M}_{2,3}(G)$
and $\vec{M}_{2,0,3}(G)$ (see \figref{rigidity-matrices}) contain repeated variables that make the combinatorial analysis of the rank complicated.

By contrast, for the closely related $(2,2)$-sparsity-matroid
and the $(2,0,2)$-graded-sparsity matroid,
representation results are easier to obtain directly.
The results of this section are representations
of the $(2,2)$-sparsity- and $(2,0,2)$-graded-sparsity matroids which are \emph{natural} in the sense
that the matrices obtained have the same dimensions at the corresponding rigidity matrices and
non-zero entries at the same positions.
The $(2,2)$-sparsity-matroid case is due to Whiteley \cite{Whiteley:1988p137},
but we include it here for completeness.

In the rest of this section, we give precise definitions of generic representations
of matroids and then prove our representation results for the $(2,2)$-sparsity
and $(2,0,2)$-graded-sparsity matroids.

\paragraph{The generic rank of a matrix.}  The matrices we define in this paper have
as their non-zero entries \emph{generic variables}, or formal polynomials over
$\mathbb{R}$ or $\mathbb{C}$ in generic variables.  We define
such a matrix $\vec M$ is to be a \emph{generic matrix}, and its \emph{generic rank}
is given by the largest number $r$
for which $\vec M$ has an $r\times r$ matrix minor with a determinant that is formally non-zero.

Let $\vec M$ be a generic matrix in $m$ generic variables $x_1,\ldots, x_m$, and
let $\vec v=(v_i)\in \mathbb{R}^m$ (or $\mathbb{C}^m$).  We define a \emph{realization $\vec M(\vec v)$ of $\vec M$} to be the matrix obtained by replacing the variable $x_i$ with the corresponding number $v_i$.  A
vector $\vec v$ is defined to be a \emph{generic point} if the rank of $\vec M(\vec v)$ is equal to the
generic rank of $\vec M$; otherwise $\vec v$ is defined to be a \emph{non-generic} point.

We will make extensive use of the following well-known facts from algebraic geomety (see, e.g., \cite{cox:little:oshea:iva:1997}):
\begin{itemize}
\item The rank of a generic matrix $\vec M$ in $m$ variables
is equal to the maximum over $\vec v\in \mathbb{R}^m$ ($\mathbb{C}^m$) of the rank of all
realizations $\vec M(\vec v)$.
\item The set of non-generic points of a generic matrix $\vec M$ is an
algebraic subset of $\mathbb{R}^m$ ($\mathbb{C}^m$).
\item The rank of a generic matrix $\vec M$ in $m$ variables is at least
as large as the rank of any specific realization $\vec M(\vec v)$.
\end{itemize}

\paragraph{Generic representations of matroids.}
A \emph{matroid} $\mathcal{M}$ on a ground set $E$ is a combinatorial structure
that captures properties of linear independence.  Matroids have many
equivalent definitions, which may be found in a monograph such as \cite{oxley:matroid}.  For our
purposes, the most convenient formulation is in terms of \emph{bases}: a matroid $\mathcal{M}$
on a finite ground set $E$ is presented by its bases $\mathcal{B}\subset 2^E$, which
satisfy the following properties:
\begin{itemize}
\item The set of bases $\mathcal{B}$ is not empty.
\item All elements $B\in \mathcal{B}$ have the same cardinality, which is the \emph{rank}
of $\mathcal{M}$.
\item For any two distinct bases $B_1,B_2\in \mathcal{B}$, there are elements
$e_1\in B_1-B_2$ and $e_2\in B_2$ such that $B_2+ \{ e_1 \}-\{ e_2\}\in \mathcal{B}$.
\end{itemize}

It is shown in \cite{pebblegame} that the set of $(2,2)$-graphs form the bases of a
matroid on the set of edges of $K_n^2$, the complete graph with edge multiplicity $2$.  In
\cite{graded} we proved that the set of looped-$(2,2)$ graphs forms a matroid
on the set of edges of $K_n^{2,2}$ a complete graph with edge multiplicity $2$ and $2$ distinct loops
on every vertex.

Let $\mathcal{M}$ be a matroid on ground set $E$.  We define a generic matrix $\vec M$ to be
a \emph{generic representation of $\mathcal{M}$} if:
\begin{itemize}
\item There is a bijection between the rows of $\vec M$ and the ground set $E$.
\item A subset of rows of $\vec M$ is attains the rank of the matrix $\vec M$
if and only if the corresponding subset of $E$ is a basis of $\mathcal{M}$.
\end{itemize}

With the definitions complete, we prove the results of this section.

\paragraph{Natural representation of spanning trees.}
We begin with a standard lemma, also employed by Whiteley \cite{Whiteley:1988p137},
about the  linear representability of the well-known spanning tree matroid.
\begin{figure}[htbp]
\centering
\includegraphics[width=0.48\textwidth]{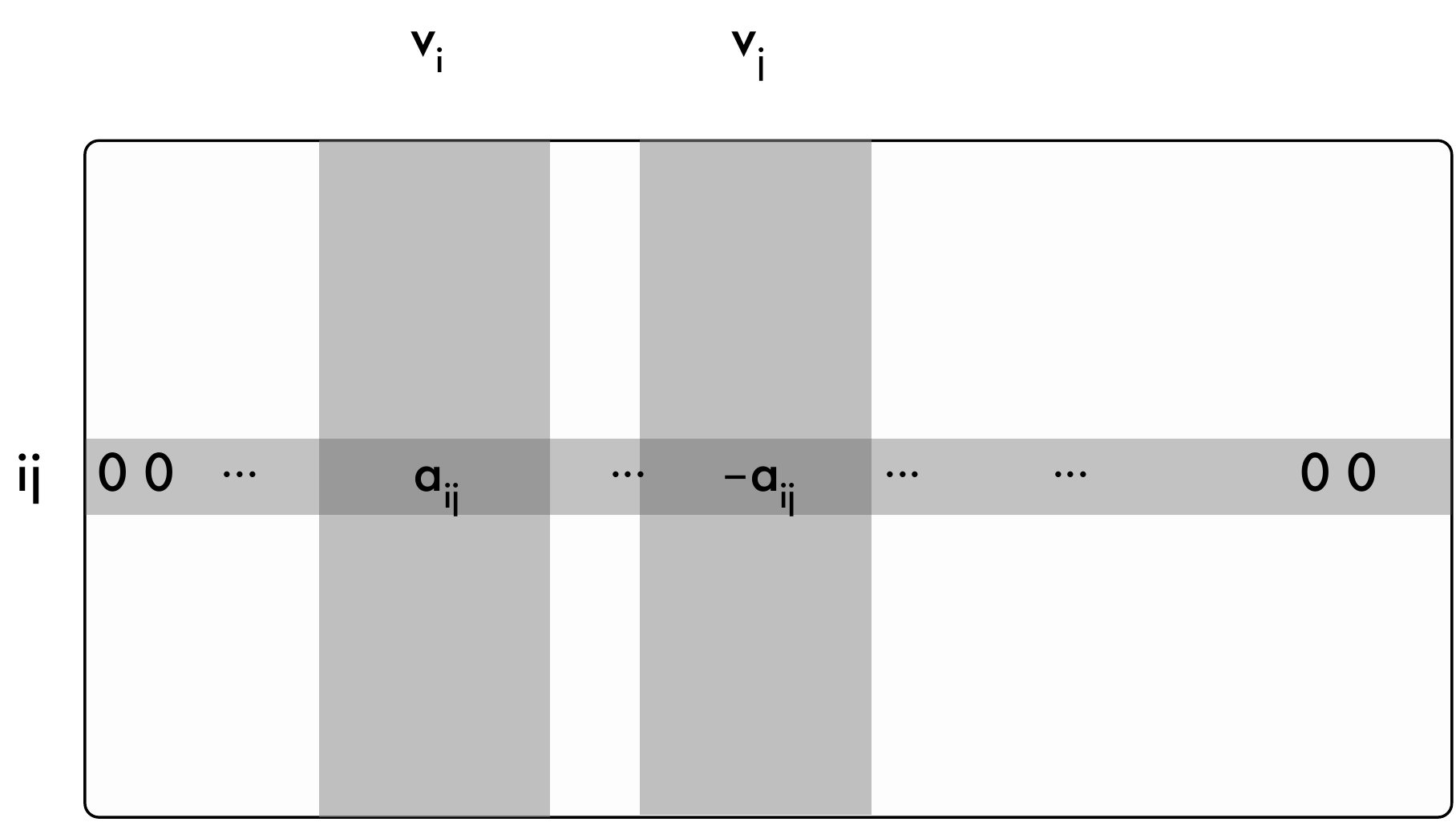}
\includegraphics[width=0.48\textwidth]{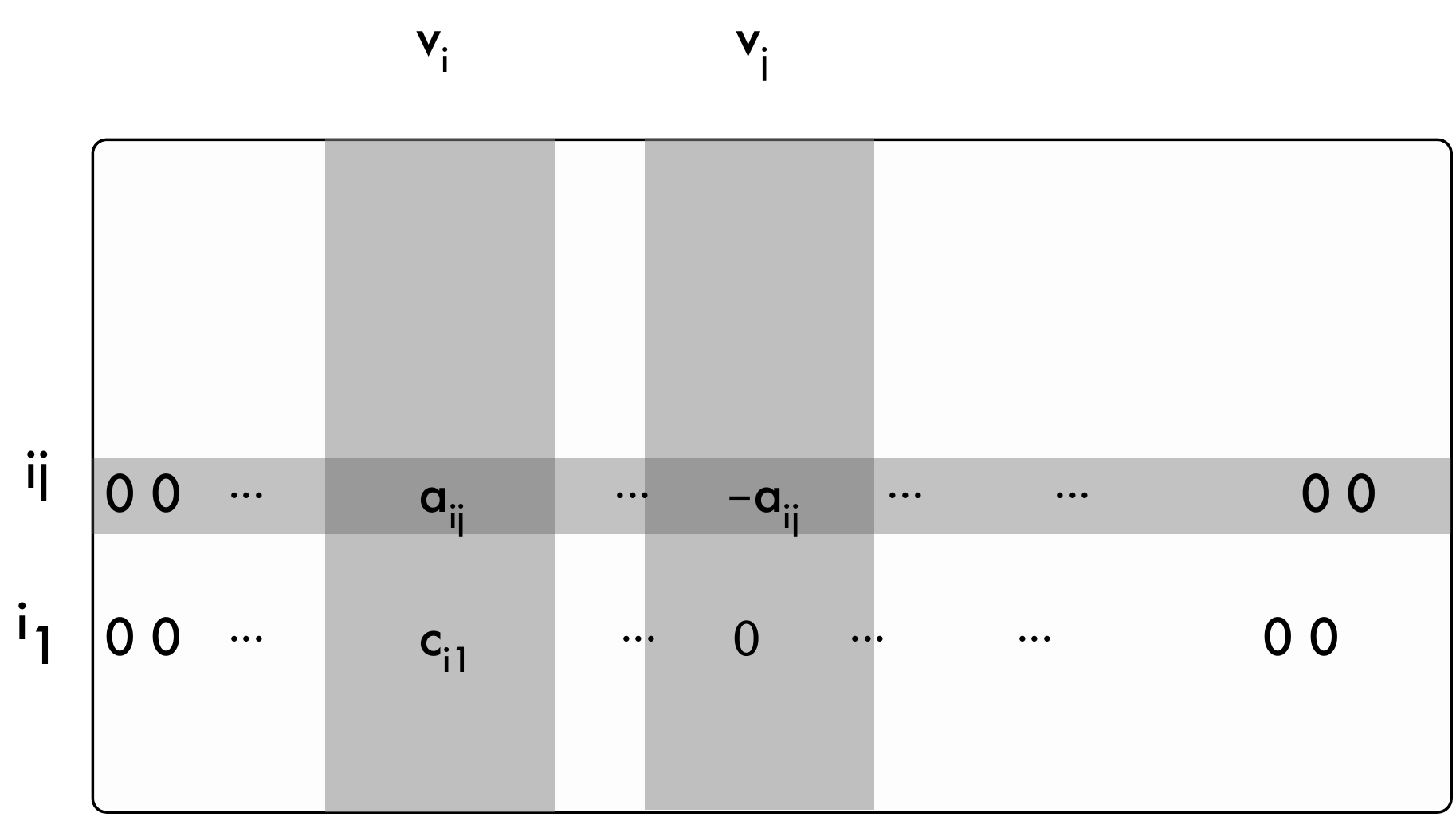}
\caption{The pattern of the matrices for trees and looped forests: (a) $\vec M_{1,1}(G)$; (b) $\vec M_{1,0,1}(G)$.}
\label{fig:m11}
\end{figure}

Let $G$ be a graph.
We define the matrix $\vec M_{1,1}(G)$ to have
one column for each vertex $i\in V$
and one row for each edge $ij\in E$.  The row $ij$ has zeros in
the columns not associated with $i$ or $j$, a generic variable $a_{ij}$
in the column for vertex $i$ and $-a_{ij}$ in the column
for vertex $j$.  \figref{m11}(a) illustrates the pattern.

We define $\vec M^{\bullet}_{1,1}(G)$ to be the matrix obtained from $\vec M_{1,1}(G)$
by dropping any column.  \lemref{treedet}  shows that the ambiguity of the column to
drop poses no problem for our purposes.

\begin{lemma}\lemlab{treedet}
Let $G$ be a graph on $n$ vertices and $m=n-1$ edges.  If $G$ is a tree, then
\[ \det\left(\vec M_{1,1}^{\bullet}(G)\right) = \pm\prod_{ij\in E(G)} a_{ij}. \]
Otherwise $\det\left(\vec M_{1,1}^{\bullet}(G)\right) = 0$.
\end{lemma}
See \cite[solution to Problem 4.9]{lovasz:problems} for the proof.

\paragraph{Natural representation of looped forests.}
In the setting of looped graphs, the object corresponding
to a spanning tree is a forest in which every connected component
spans exactly one loop.  We define such a graph to be a \emph{looped forest}.
Looped forests are special cases of the \emph{map-graphs}
studied in our papers \cite{maps,colors,graded}, which develop their combinatorial and
matroidal properties.

Let $G$ be a looped graph and
define the matrix $\vec M_{1,0,1}(G)$ to have
one column for each vertex $i\in V$.  Each edge has a row
corresponding to it with the same pattern as in $\vec M_{1,1}(G)$.  Each loop
$i_j$ has a row corresponding to it with a variable $c_{i_j}$ in the column corresponding
to vertex $i$ and zeros elsewhere.  \figref{m11}(b) shows the pattern. \lemref{treedet} generalizes to the following.

\begin{lemma}\lemlab{mapdet}
Let $G$ be a looped graph on $n$ vertices and $c+m=n$ edges and loops.  If $G$ is a looped forest, then
\[ \det\left( \vec M_{1,0,1}(G)\right) = \pm\left(\prod_{\text{\rm edges}\, ij\in E(G)} a_{ij}\right)\cdot
\left( \prod_{\text{\rm loops}\, i_j\in E(G)} c_{i_j}\right) \]  Otherwise
$\det\left(\vec M_{1,0,1}^{\bullet}(G)\right) = 0$.
\end{lemma}
\begin{proof}
By the hypothesis of the lemma, $ \vec M_{1,0,1}(G)$ is $n\times n$, so its determinant is well-defined.

If $G$ is not a looped forest, then it has a vertex-induced subgraph $G'$ on $n'$ vertices spanning
at least $n'+1$ edged and loops.  The sub-matrix induced by the rows corresponding to edges and loops in $G'$ has at least $n'+1$ rows by at most $n'$ columns that are not all zero.

If $G$ is a looped forest then $\vec M_{1,0,1}(G)$ can be arranged to have a block diagonal
structure.  Partition the vertices according to the $k\ge 1$ connected components $G_1, G_2,\ldots, G_k$
and arrange the columns so that $V(G_1), V(G_2), \ldots, V(G_k)$ appear in order.  Then arrange the
rows so that the $E(G_i)$ also appear in order.  Thus the lemma follows from proving that if $G$ is a tree
with a loop on vertex $i$ we have
\[
\det\left( \vec M_{1,0,1}(G)\right) = \pm c_{i_1}\cdot \left(\prod_{\text{edges}\, ij\in E(G)} a_{ij}\right)
\]
since we can multiply the determinants of the sub-matrices corresponding to the connected components.

To complete the proof, we expand the determinant along the row corresponding to the loop $i_1$.  Since it
has one non-zero entry, we have
\[
\det\left( \vec M_{1,0,1}(G)\right) = \pm c_{i_1} \det\left( \vec M_{1,0,1}(G)[A,B]\right)
\]
where $A$ is the set of rows correspond to the $n-1$ edges of $G$ and $B$ is the set of columns
corresponding to all the vertices of $G$ except for $i$.  Since $\vec M_{1,0,1}(G)[A,B]$ has the
same form at $\vec M_{1,1}^{\bullet}(G-\{ i_j \})$ the claimed determinant formula follows from
\lemref{treedet}.
\end{proof}

\paragraph{The $(2,2)$-sparsity-matroid.}
Let $G$ be a graph.  We define the matrix $\vec M_{2,2}(G)$ to have
two columns for each vertex $i\in V$
and one row for each edge $ij\in E$.  The row $ij$ has zeros in
the columns not associated with $i$ or $J$, variables $(a_{ij},b_{ij})$
in the columns for vertex $i$ and $(-a_{ij},-b_{ij})$ in the columns
for vertex $j$.  \figref{m22} illustrates the pattern.

\begin{figure}[htbp]
\centering
\includegraphics[width=0.48\textwidth]{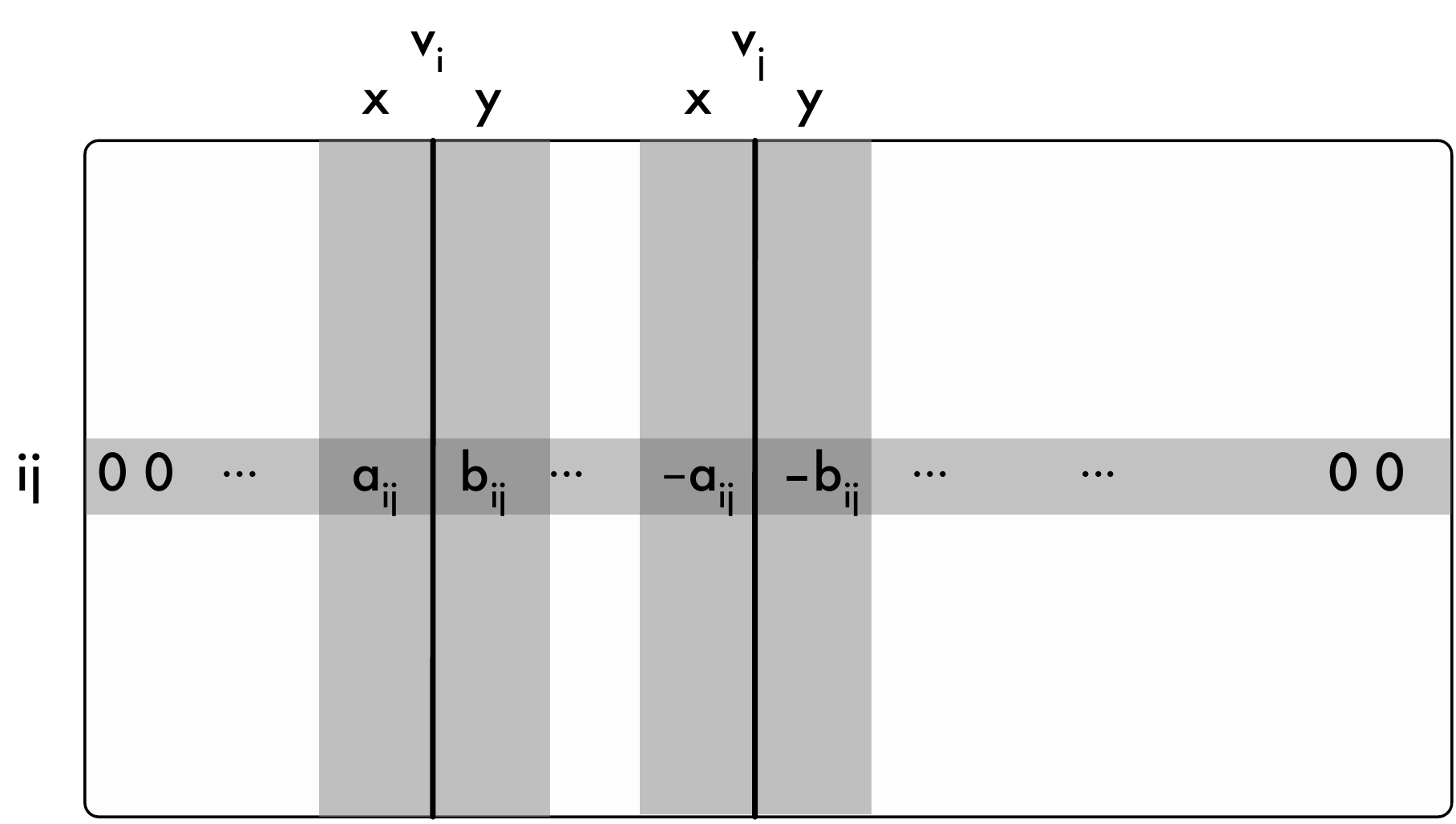}
\includegraphics[width=0.48\textwidth]{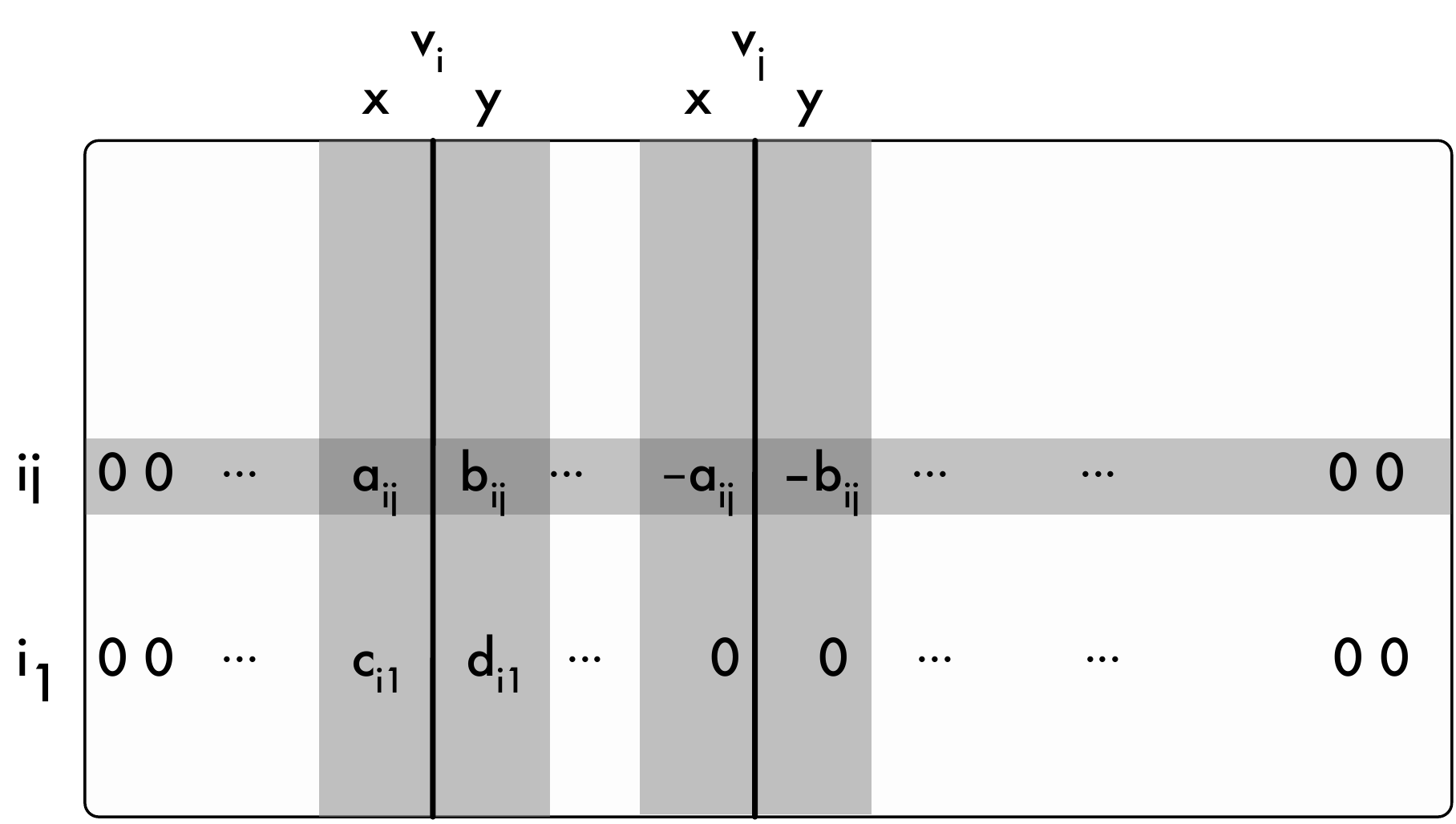}

\caption{The pattern of the matrices for $(2,2)$-graphs and looped-$(2,2)$ graphs: (a) $\vec M_{2,2}(G)$; (b) $\vec M_{2,0,2}(G)$.}
\label{fig:m22}
\end{figure}
\begin{lemma}[Whiteley \cite{Whiteley:1988p137}]\lemlab{22matrix}
The matrix $\vec M_{2,2}(K_n^2)$ is
a generic representation of the $(2,2)$-sparsity matroid.
\end{lemma}
The proof, which can be found in \cite{Whiteley:1988p137},
is essentially the same as that used to prove the Matroid Union Theorem
for linearly representable matroids (e.g., \cite[Prop. 7.6.14]{brylawski:constructionsSurvey}).

\paragraph{The $(2,0,2)$-graded-sparsity matroid.}
Let $G$ be a looped graph and define the matrix $\vec M_{2,0,2}(G)$ to have two columns for
each vertex, one row for each edge or self-loop.  The rows for the edges are the same as
in $\vec M_{2,2}(G)$.  The row for a self-loop $i_{j}$ (the $j$th loop on vertex $i$)
has variables $(c_{i_{j}},d_{i_{j}})$ in the columns for vertex $i$ and zeros elsewhere.
(See \figref{m22}(b).)

\begin{lemma}\lemlab{202matrix}
The matrix $\vec M_{2,0,2}(K_2^{2,2})$ is a generic representation of
the $(2,0,2)$-graded-sparsity-matroid.
\end{lemma}
\begin{proof}
We need to show that
if $G$ has $n$ vertices, and $m+c=2n$ edges and loops, then the generic rank of
$\vec M_{2,0,2}(G)$ is $2n$ if and only if $G$ is a looped-$(2,2)$ graph.

Since $\vec M_{2,0,2}(G)$ is square, we expand the determinant around the $a_\cdot$
columns with the generalized Laplace expansion to get:
\[
\sum \pm \det\left(\vec M_{2,0,2}(G)[A,X]\right)\cdot \det\left(\vec M_{2,0,2}(G)[B,Y]\right)
\]
where the sum is over all complementary sets of $n$ rows $X$ and $Y$.  Since each smaller
determinant has the form of $\vec M_{1,0,1}(G)$ from \lemref{mapdet}, the sum has
a non-zero term if and only if $G$ is the edge-disjoint union of two looped forests.
Any non-zero term is a multilinear monomial that cannot generically cancel with any
of the others, implying that the generic rank of $\vec M_{2,0,2}(G)$ is $2n$ if and only
if $G$ is the disjoint union of two looped forests.

The lemma then follows from the main theorems of our papers \cite{colors,graded},
which show that $G$ admits such a decomposition if and only if $G$ is looped-$(2,2)$.
\end{proof}

\section{Direction network realization}\seclab{parallel}
A \emph{direction network} $(G,\vec d)$ is a graph $G$ together with an
assignment of a direction vector $\vec d_{ij}\in \Re^2$ to each edge.
The \emph{direction network realizability problem} \cite{Whiteley:1989p992}
is to find a realization $G(\vec p)$ of a direction network $(G,\vec d)$.

A \emph{realization} \label{realizationdef} $G(\vec p)$ of a direction network is an embedding of $G$
onto a point set $\vec p$ such that $\vec p_i-\vec p_j$ is in the direction
$\vec d_{ij}$.
In a realization $G(\vec p)$ of a direction network $(G,\vec d)$, an edge $ij$ is \emph{collapsed} if $\vec p_i=\vec p_j$.  A  realization is collapsed if all the $\vec p_i$ are the same.  A realization is \emph{faithful}\footnote{Whiteley \cite{Whiteley:1989p992} calls this condition ``proper.''} if $ij\in E$ implies that $\vec p_i\neq \vec p_j$.  In other words, a faithful parallel realization has no collapsed edges.

In this section, in preparation for the main result, we give a new derivation of the Parallel Redrawing Theorem
of Whiteley. 

\parallelthm

\paragraph{Roadmap.} Here is an outline of the proof.
\begin{itemize}
\item We formally define the direction network realization problem
as a linear system $\vec P(G,\vec d)$ and prove that its generic rank
is equivalent to that of $\vec M_{2,2}(G)$.  (\lemref{matrixparallel} and \lemref{parallelrank}.)
\item  We show that if a
solution to the realization problem $\vec P(G,\vec d)$
collapses an edge $vw$, the solution space is equivalent
to the solution space of $\vec P_{vw}(G,\vec d)$, a linear
system in which $\vec p_v$ is replaced with $\vec p_w$. The
combinatorial interpretation of this algebraic result is that the
realizations of $(G/vw,\vec d)$ are in bijective correspondence with those of
$(G,\vec d)$. (\lemref{contractreplace} and \corref{collapsecontract}.)
\item We then state and prove a \emph{genericity condition} for direction
networks $(G,\vec d)$ where $G$ is $(2,2)$-sparse and has $2n-3$ edges:
the set of $\vec d$ such that $(G,\vec d)$ and all contracted networks $(G/ij,\vec d)$
is open and dense in $\mathbb{R}^{2m}$. (\lemref{parallelgeneric}.)
\item The final step in the proof is to show that for a Laman graph,
if there is a collapsed edge in a generic realization, then the whole
realization is collapsed by the previous steps and obtain a contradiction. (Proof of \theoref{parallel}.)
\end{itemize}

\subsection{Direction network realization as a linear system}
Let $(G,\vec d)$ be a direction network.  We define the linear system
$\vec P(G,\vec d)$ to be
\begin{eqnarray}
\iprod{\vec p_i-\vec p_j}{\vec d_{ij}^\perp} = 0 & \text{for all $ij\in E$}\eqlab{edges}
\end{eqnarray}
where the $\vec p_i$ are the unknowns.  From the definition of a realization (p. \pageref{realizationdef}, above the statement of \theoref{parallel}), every
realization $G(\vec p)$ of $(G,\vec d)$, $\vec p$ is a solution of $\vec P(G,\vec d)$.

If the entries of $\vec d$ are generic variables, then the solutions to $\vec P(G,\vec d)$
are polynomials in the entries of $\vec d$.  We start by describing  $\vec P(G,\vec d)$
in matrix form.

\begin{lemma}\lemlab{matrixparallel}
Let $(G,\vec d)$ be a direction network.  Then the solutions $\vec p$ of
the system $\vec P(G,\vec d)$ are solutions to the matrix equation
\[
\vec M_{2,2}(G)\vec p = \vec 0
\]
\end{lemma}
\begin{proof}
Bilinearity of the inner product implies that \eqref{edges} is equivalent to
\[
\iprod{\vec p_i}{\vec d_{ij}^\perp} + \iprod{\vec p_j}{-\vec d_{ij}^\perp} = 0
\]
which in matrix form is $\vec M_{2,2}(G)$.
\end{proof}

The matrix form of $\vec P(G,\vec d)$ leads to an immediate connection to
the $(2,2)$-sparsity-matroid.

\begin{lemma}\lemlab{parallelrank}
Let $G$ be a graph on $n$ vertices with $m\le 2n-2$ edges.
The generic rank of $\vec P(G,\vec d)$
(with the $2n$ variables in $\vec p=(\vec p_1,\cdots, \vec p_n)$ as the unknowns) is $m$ if and only if $G$ is
$(2,2)$-sparse.  In particular, the rank is $2n-2$ if and only if $G$ is a $(2,2)$-graph.
\end{lemma}
\begin{proof}
Apply \lemref{matrixparallel} and then \lemref{22matrix}.
\end{proof}

An immediate consequence of \lemref{parallelrank} that we will use frequently is the following.
\begin{lemma}\lemlab{genericdirections}
Let $G$ be $(2,2)$-sparse.  Then the set of edge direction assignments
$\vec d\in \mathbb{R}^{2m}$ such that
the direction network realization system $\vec P(G,\vec d)$ has rank
$m$ is the (open, dense) complement of an algebraic subset of $\mathbb{R}^{2m}$.
\end{lemma}
\begin{proof}
By \lemref{parallelrank} any $\vec d\in \mathbb{R}^{2m}$ for which the rank drops is a
common zero of the $m\times m$ minors of the generic matrix
$\vec M_{2,2}(G)$, which are polynomials.
\end{proof}

Because of \lemref{genericdirections}, when we work with $\vec P(G,\vec d)$ as a system with numerical
directions, we may select directions $\vec d\in \mathbb{R}^{2m}$ such that $\vec P(G,\vec d)$ has full rank when $G$ is $(2,2)$-sparse.  We use this fact repeatedly below.

\paragraph{Translation invariance of $\vec P(G,\vec d)$.}
Another simple property is that solutions to the system $\vec P(G,\vec d)$ are preserved
by translation.
\begin{lemma}\lemlab{translation}
The space of solutions to the system $\vec P(G,\vec d)$ is preserved by translation.
\end{lemma}
\begin{proof}
Let $\vec t$ be a vector in $\Re^2$.  Then $\iprod{(\vec p_i+\vec t)-(\vec p_j+\vec t)}{\vec d_{ij}^\perp}=\iprod{\vec p_i-\vec p_j}{\vec d_{ij}^\perp}$.
\end{proof}

\subsection{Realizations of direction networks on $(2,2)$-graphs}
There is a simple characterization of realizations of generic direction networks on
$(2,2)$-graphs: they are all collapsed.

\begin{lemma}\lemlab{22parallel}
Let $G$ be a $(2,2)$-graph on $n$ vertices,
and let $\vec d_{ij}$ be directions such that the system $\vec P(G,\vec d)$ has rank $2n-2$.  (This is possible by \lemref{genericdirections}.)
Then the (unique up to translation)  realization of $G$ with directions
$\vec d_{ij}$ is collapsed.
\end{lemma}
\begin{proof}
By hypothesis the system $\vec P(G,\vec d)$ is homogeneous of rank $2n-2$.  Factoring out
translations by moving the variables giving associated with
$\vec p_1$ to the right, we have a unique solution
for each setting of the value of $\vec p_1$.
Since a collapsed realization satisfies the system, it is the only one.
\end{proof}

\subsection{Realizations of direction networks on Laman graphs}
In the rest of this section we complete the proof of \theoref{parallel}.

\paragraph{The contracted direction network realization problem.}
Let $(G,\vec d)$ be a direction network, with realization system $\vec P(G,\vec d)$, and let
$vw$ be an edge of $G$.  We define the \emph{$vw$-contracted realization system}
$\vec P_{vw}(G,\vec d)$ to be the linear system obtained by replacing
$\vec p_v$ with $\vec p_w$ in $\vec P(G,\vec d)$.

\paragraph{Combinatorial interpretation of $\vec P_{vw}(G)$.}
We relate $\vec P(G/vw,\vec d)$ and $\vec P_{vw}(G,\vec d)$ in the following lemma.

\begin{lemma}\lemlab{contractreplace}
Let $(G,\vec d)$ be a generic direction network.  Then for any edge $vw$ the system
$\vec P_{vw}(G,\vec d)$ is the same as the system $\vec P(G/vw,\vec d)$, and the
generic rank of $\vec P_{vw}(G,\vec d)$ is the same as that of $\vec M_{2,2}(G/vw)$.
\end{lemma}
\begin{proof}
By definition, in the system $\vec P_{vw}(G,\vec d)$:
\begin{itemize}
\item The point $\vec p_v$ disappears
\item Every occurrence of $\vec p_v$ is replaced with $\vec p_w$
\end{itemize}
Combinatorially, this corresponds to contracting over the edge $vw$ in $G$, which shows
that $\vec P_{vw}(G,\vec d)$ is the same system as $\vec P(G/vw,\vec d)$.  An
application of \lemref{parallelrank} to $\vec P(G/vw,\vec d)$ shows that its
rank is equivalent to that of $\vec M_{2,2}(G/vw)$.
\end{proof}

Since the replacement of $\vec p_v$ with $\vec p_w$ is the same as setting $\vec p_v=\vec p_w$, we have the
following corollary to \lemref{contractreplace}.

\begin{cor}\corlab{collapsecontract}
Let $(G,\vec d)$ be a direction network and $ij$ an edge in $G$.
If in every solution $\vec p$ of
$\vec P(G,\vec d)$,  $\vec p_i=\vec p_j$, then $\vec p$ is a solution to $\vec P(G,\vec d)$ if and only if $\vec p'$
obtained by dropping $\vec p_i$ from $\vec p$ is a solution to $\vec P(G/ij, \vec d)$.
\end{cor}

\paragraph{A genericity condition.}
The final ingredient we need is the following genericity condition.

\begin{lemma}\lemlab{parallelgeneric}
Let $G$ be a Laman graph on $n$ vertices.  Then the set of directions
$\vec d\in \mathbb{R}^{2m}$ such that:
\begin{itemize}
\item The system $\vec P(G,\vec d)$ has rank $2n-3$
\item For all edges $ij\in E$, the system $\vec P(G/ij,\vec d)$ has rank $2(n-1) - 2$
\end{itemize}
is open and dense in $\mathbb{R}^{2m}$.
\end{lemma}
\begin{proof}
By \lemref{22contract} all the graphs $G/ij$ are $(2,2)$-graphs and since $G$ is Laman, all the
graphs appearing in the hypothesis are $(2,2)$-sparse, so we may apply \lemref{genericdirections}
to each of them separately.  The set of $\vec d$ failing the requirements of the lemma is thus the union of finitely many closed algebraic sets in $\mathbb{R}^{2m}$ of measure zero.  Its complement is open and dense, as required.
\end{proof}

\paragraph{Proof of \theoref{parallel}.}

We first assume that $G$ is not Laman.  In this case it has an edge-induced subgraph $G'$
that is a $(2,2)$-graph by results of \cite{pebblegame}.  This means that
for generic directions $\vec d$, the system $\vec P(G,\vec d)$
has a subsystem corresponding to $G'$ to which \lemref{22parallel} applies.  Thus any
realization of $(G,\vec d)$ has a collapsed edge.

For the other direction, we assume, without loss of generality, that $G$ is a Laman graph.
We select directions $\vec d$ meeting the criteria of
\lemref{parallelgeneric} and consider the direction network $(G,\vec d)$.

Since $\vec P(G,d)$ has $2n$ variables and rank $2n-3$, we move $\vec p_1$ to the right to remove translational symmetry and one other variable, say, $x_2$, where $\vec p_2=(x_2,y_2)$.  The system
has full rank, so for each setting of $\vec p_1$ and $x_2$ we obtain a unique solution.  Set $\vec p_1=(0,0)$ and $x_2=1$ to get a solution $\hat{\vec p}$
of $\vec P(G,\vec d)$ where $\vec p_1\neq \vec p_2$.

We claim that $G(\hat{\vec p})$ is faithful.  Supposing the contrary, for a contradiction,
we assume that some edge $ij\in E$ is collapsed in $G(\hat{\vec p})$.  Then the equation
$\vec p_i = \vec p_j$ is implied by $\vec P(G,\vec d)$.  Applying \corref{collapsecontract},
we see that after removing $\hat{\vec p_i}$ from $\hat{\vec p}$, we obtain a solution to $\vec P(G/ij,\vec d)$.  But then by \lemref{22contract}, $G/ij$ is a $(2,2)$-graph.  Because $\vec d$ was selected (using \lemref{parallelgeneric})
so that $\vec P(G/ij,\vec d)$ has full rank, \lemref{22parallel} applies to $(G/ij,\vec d)$,
showing that every edge is collapsed in $G(\hat{\vec p})$.  We have now arrived
at a contradiction: $G$ is connected, and by construction $\vec p_1\neq \vec p_2$, so some edge
is not collapsed in $G(\hat{\vec p})$. $\qed$

\paragraph{Remarks on genericity.} The proof of \theoref{parallel}
shows why each of the two conditions in \lemref{parallelgeneric}
are required.  The first, that $\vec P(G,\vec d)$ have full rank, ensures that there is a unique solution up to translation.
The second condition, that for each edge $ij$ the system $\vec P(G,\vec d)$ has full rank, rules out sets of directions that
are only realizable with collapsed edges.

The second condition in the proof is necessary by the following example:
let $G$ be a triangle and assign two of its edges the horizontal
direction and the other edge the vertical direction.  It is easy to check that the resulting $\vec P(G,\vec d)$ has full rank,
but it is geometrically evident that the edges of a non-collapsed triangle require either one or three directions.  This example
is ruled out by the contraction condition in \lemref{parallelgeneric}, since contracting the vertical edge results in a
rank-deficient system with two vertices and two copies of an edge in the same direction.

\section{Direction-slider network realization}\seclab{sliderparallel}
A \emph{direction-slider network} $(G,\vec d,\vec n,\vec s)$ is a looped graph, together with assignments of:
\begin{itemize}
\item A direction $\vec d_{ij}\in \mathbb{R}^2$ to each edge $ij\in E$.
\item A \emph{slider}, which is an affine line $\iprod{\vec n_{i_j}}{\vec x}=s_{i_j}$
in the plane, to each loop $i_j\in E$.
\end{itemize}

A \emph{realization} $G(\vec p)$ of a direction-slider network is an embedding of $G$
onto the point set $\vec p$ such that:
\begin{itemize}
\item Each edge $ij$ is drawn in the direction $\vec d_{ij}$.
\item For each loop $i_j$ on a vertex $i$, the point $\vec p_i$
is on the line $\iprod{\vec n_{i_j}}{\vec x}=s_{i_j}$.
\end{itemize}

As in the definitions for direction networks in the previous section, an edge $ij$ is \emph{collapsed}
in a realization $G(\vec p)$ if $\vec p_i=\vec p_j$.  A realization $G(\vec p)$ is \emph{faithful}
if none of the edges of $G$ are collapsed.

The main result of this section is:
\sliderparallelthm

\paragraph{Roadmap.}  The approach runs along the lines of previous section.  However,
because the system $\vec S(G,\vec d,\vec n, \vec s)$ is inhomogeneous, we obtain
a contradiction using unsolvability instead of a unique collapsed realization.
The steps are:
\begin{itemize}
\item  Formulate the direction-slider realization problem as a linear system and relate the rank of the parallel sliders realization system to the
representation of the $(2,0,2)$-sparsity-matroid to show the generic rank of the realization
system is given by the rank of the graph $G$ in the $(2,0,2)$-matroid. (\lemref{parallelsliderrank})
\item Connect graph theoretic contraction over an edge $ij$ to the edge being collapsed in all realizations of the direction-slider network: show that when
$\vec S(G,\vec d,\vec n, \vec s)$ implies that some edge $ij$ is collapsed in all realizations means
that it is equivalent to $\vec S(G/ij,\vec d,\vec n, \vec s)$. (\lemref{slidercontractreplace} and \corref{slidercollapsecontract})
\item Show that for looped graphs with combinatorially independent edges and one too
many loops, the system $\vec S(G/ij,\vec d,\vec n, \vec s)$ is generically not solvable.
(\lemref{genericsliderdirections}).
\item Show that if $G$ is looped-Laman, then there are generic directions and sliders for $\vec M_{2,0,2}(G)$ so that the contraction of any edge leads to an unsolvable system. (\lemref{slidergenericity}.)
\item Put the above tools together to show that for a looped-Laman graph, the realization problem is generically solvable, and the (unique solution) does not collapse any edges.
\end{itemize}

\subsection{Direction-slider realization as a linear system.}
Let $(G,\vec d,\vec n,\vec s)$ be a direction-slider network.  We define the system of equations
$\vec S(G,\vec d,\vec n, \vec s)$ to be:
\begin{eqnarray}
\eqlab{s1}	\iprod{\vec p_i-\vec p_j}{\vec d_{ij}^\perp} = 0 & \text{for all edges $ij\in E$} \\
\eqlab{s2}	\iprod{\vec p_i}{\vec n_{i_j}} = s_{i_j} & \text{for all loops $i_j\in E$}
\end{eqnarray}
From the definition, it is immediate that the realizations of $(G,\vec d,\vec n,\vec s)$
are exactly the solutions of $\vec S(G,\vec d,\vec n, \vec s)$.  The matrix form
of $\vec S(G,\vec d,\vec n, \vec s)$ gives the connection to the $(2,0,2)$-sparsity matroid.

\begin{lemma}\lemlab{parallelslidermatrix}
Let $(G,\vec d,\vec n,\vec s)$ be a direction slider network.  The solutions to the
system $\vec S(G,\vec d,\vec n, \vec s)$ are exactly the solutions to the matrix equation
\[\vec M_{2,0,2}(G)\vec p = (\vec 0, \vec s)^{\mathsf{T}}\]
\end{lemma}
\begin{proof}
Similar to the proof of \lemref{matrixparallel} for the edges of $G$. The slider
are already in the desired form.
\end{proof}

As a consequence, we obtain the following two lemmas.

\begin{lemma}\lemlab{parallelsliderrank}
Let $G$ be a graph on $n$ vertices with $m\le 2n$ edges.
The generic rank of $\vec S(G,\vec d,\vec n, \vec s)$
(with the $\vec p_i$ as the $2n$ unknowns) is $m$ if and only if $G$ is
$(2,0,2)$-sparse.  In particular, it is $2n$ if and only if $G$ is a looped-$(2,2)$ graph.
\end{lemma}
\begin{proof}
Apply \lemref{parallelslidermatrix} and then \lemref{202matrix}.
\end{proof}

We need, in addition, the following result on when $\vec S(G,\vec d,\vec n, \vec s)$
has no solution.
\begin{lemma}\lemlab{genericsliderdirections}
Let $G$ be a looped-$(2,2)$ graph and let $G'$ be obtained
from $G$ by adding a single loop $i_j$ to $G$.
Then the set of edge direction assignments and slider
lines $(\vec d,\vec n, \vec s)\in \mathbb{R}^{2m+3c}$ such that
the direction-slider network realization system $\vec S(G',\vec d,\vec n, \vec s)$
has no solution
is the (open, dense) complement of an algebraic subset of $\mathbb{R}^{2m+3c}$.
\end{lemma}
\begin{proof}
By \lemref{parallelslidermatrix} and \lemref{parallelsliderrank}, the solution $\vec p=\hat{\vec p}$
to the generic matrix equation
\[
\vec M_{2,0,2}(G)\vec p =  (\vec 0, \vec s)^{\mathsf{T}}
\]
has as its entries non-zero formal
polynomials in the entries of $\vec d$, $\vec n$, and $\vec s$.  In particular,
the entries of $\hat{\vec p}_i$ are non-zero.  This implies that for the equation
\[
\vec M_{2,0,2}(G')\vec p =  (\vec 0, \vec s)^{\mathsf{T}}
\]
to be solvable, the solution will have to be $\hat{\vec p}$, and  $\hat{\vec p}_i$ will have to satisfy the additional equation
\[
\iprod{\vec n_{i_j}}{\hat{\vec p_i}} = s_{i_j}
\]
Since the entries of $\vec n_{i_j}$ and $s_{i_j}$ are generic and don't appear at in
$\hat{\vec p}_i$, the system $\vec S(G',\vec d,\vec n, \vec s)$ is solvable only when either
the rank of $\vec M_{2,0,2}(G)$ drops, which happens only for closed algebraic subset of
$\mathbb{R}^{2m+3c}$ or when $\vec n_{i_j}$ and $s_{i_j}$ satisfy the above equation, which
is also a closed algebraic set.  (Geometrically, the latter condition says that the line of the
slider corresponding to the loop $i_j$ is in the pencil of lines through $\hat{\vec p_i}$.)
\end{proof}

\paragraph{Contracted systems.}
Let $vw\in E$ be an edge.  We define $\vec S_{vw}(G,\vec d,\vec n, \vec s)$,
the contracted realization system, which is obtained by
replacing $\vec p_v$ with $\vec p_w$ in $\vec S(G,\vec d,\vec n, \vec s)$.
The contracted system has two fewer variables and one fewer equation (corresponding to
the edge $vw$).

The proof of \lemref{contractreplace} is identical to the proof of the analogous result for direction-slider networks.
\begin{lemma}\lemlab{slidercontractreplace}
Let $(G,\vec d, \vec n,\vec s)$ be a generic direction-slider network.  Then for any edge $vw$ the system
$\vec P_{vw}(G,\vec d)$ is the same as the system $\vec P(G/vw,\vec d, \vec n,\vec s)$, and the
generic rank of $\vec P_{vw}(G,\vec d, \vec n,\vec s)$ is the same as that of $\vec M_{2,0,2}(G/vw)$.
\end{lemma}
The following is the direction-slider analogue of \corref{collapsecontract}.
\begin{cor}\corlab{slidercollapsecontract}
Let $(G,\vec d, \vec n,\vec s)$ be a direction-slider network and $ij$ an edge in $G$.
If in all solutions $\vec p$ of $\vec P(G,\vec d, \vec n,\vec s)$
$\vec p_i=\vec p_j$,
then $\vec p$ is a solution to $\vec P(G,\vec d, \vec n,\vec s)$ if and only if $\vec p'$ obtained by
dropping $\vec p_i$ from $\vec p$ is a solution to $\vec P(G/ij, \vec d, \vec n,\vec s)$.
\end{cor}

\paragraph{A genericity condition.}  The following lemma, which is the counterpart of
\lemref{parallelgeneric}, captures genericity for direction-slider networks.
\begin{lemma}\lemlab{slidergenericity}
Let $G$ be a looped-Laman subgraph.  The set of directions and slider lines such that:
\begin{itemize}
\item The system $\vec S(G,\vec d,\vec n, \vec s)$ has rank $2n$ (and thus has a unique solution)
\item For all edges $ij\in E$, the system $\vec S(G/ij,\vec d,\vec n, \vec s)$ has no solution
\end{itemize}
is open and dense in $\mathbb{R}^{2m+3c}$.
\end{lemma}
\begin{proof}
Because a looped-Laman graph is also a looped-$(2,2)$ graph, \lemref{202matrix}
and  \lemref{parallelsliderrank} imply that $\det (\vec M_{2,0,2}(G))$ which
is a polynomial in the entries of $\vec d$ and  $\vec n$ is not constantly zero,
and so for any values of $\vec s$, the generic
system $\vec S(G,\vec d,\vec n, \vec s)$
has a unique solution $\hat{\vec p}$ satisfying
\[
\vec M_{2,0,2}(G)\hat{\vec p} = (\vec 0, \vec s)^{\mathsf{T}}
\]
The generic directions and slider lines are the ones in the complement of the zero set
of $\det (\vec M_{2,0,2}(G))$, and the non-generic set has measure zero.

By the combinatorial \lemref{loopedlamancontract}, each edge contraction
$G/ij$ has the combinatorial form required by \lemref{slidergenericity}.  By \lemref{slidergenericity},
for each of $m$ contractions, the set of directions and slider lines
such that the contracted system $\vec S(G/ij,\vec d,\vec n, \vec s)$ is
an algebraic set of measure zero.

The proof follows from the fact that set of directions and slider
lines for which the conclusion fails is the union of a finite number of
measure-zero algebraic sets: $\det (\vec M_{2,0,2}(G))=0$ is one
non-generic set and each application of \lemref{slidergenericity}
gives another algebraic set to avoid.  Since the union of finitely many measure zero algebraic
sets is itself a measure zero algebraic set, the intersection of the complements is non-empty.
\end{proof}

\paragraph{Proof of \theoref{sliderparallel}.}
With all the tools in place, we give the proof
of our direction-slider network realization theorem.
\begin{proof}[Proof of \theoref{sliderparallel}]
If $G$ is not looped-Laman, then by \lemref{22parallel} applied on
a $(2,2)$-tight subgraph, $G$ has no faithful realization.

Now we assume that $G$ is looped-Laman.  Assign generic directions and sliders as in \lemref{slidergenericity}.
By \lemref{parallelsliderrank}, the system $\vec S(G,\vec d,\vec n, \vec s)$
has rank $2n$ and thus a unique solution.  For a contradiction, we suppose that
some edge $ij$ is collapsed.  Then by \lemref{slidercontractreplace} and
\corref{slidercollapsecontract} this system has a non-empty solution space equivalent to the
contracted system $\vec S(G/ij,\vec d,\vec n, \vec s)$.  However, since we picked the directions and sliders as in
\lemref{slidergenericity}, $\vec S(G/ij,\vec d,\vec n, \vec s)$ has no solution, leading to a contradiction.
\end{proof}

\section{Axis-parallel sliders}\seclab{xyparallelslider}
An \emph{axis-parallel direction-slider network} is a direction network in
which each slider is either vertical or horizontal.  The combinatorial
model for axis-parallel direction-slider networks is defined to be a
looped graph in which each loop is colored either red or blue, indicating
slider direction.  A \emph{color-looped-Laman graph} is a looped
graph with colored loops that is looped-Laman, and, in addition, admits a coloring
of its edges into red and blue forests so that each monochromatic tree spans
exactly one loop of its color.
Since the slider directions of an axis-parallel direction-slider network are
given by the combinatorial data, it is formally defined by the tuple
$(G,\vec d,\vec s)$.  The realization problem for axis-parallel direction-slider
networks is simply the specialization of the slider equations to $x_i=s_{i_j}$,
where $\vec p_i=(x_i,y_i)$, for vertical sliders and $y_i=s_{i_j}$ for horizontal
ones.

We prove the following extension to \theoref{sliderparallel}.

\begin{theorem}[\xysliderparallelthm][\textbf{Generic axis-parallel direction-slider network realization}]\theolab{xysliderparallel}
Let $(G,\vec d,\vec s)$ be a generic axis-parallel direction-slider network.  Then
$(G,\vec d,\vec s)$ has a (unique) faithful realization if and only if
$G$ is a color looped-Laman graph.
\end{theorem}

The proof of \theoref{xysliderparallel} is a specialization of the arguments in the
previous section to the axis-parallel setting.  The modifications we need to make are:
\begin{itemize}
\item Specialize the $(2,0,2)$-matroid realization \lemref{202matrix} to the case
where in each row coresponding to a slider $i_j\in E$ one of $c_{i_j}$ and $d_{i_j}$
is zero and the other is one.  This corresponds to the slider direction equations in
the realization system for an axis-parallel direction-slider network.
\item Specialize the genericity statement \lemref{slidergenericity}
\end{itemize}
Otherwise the proof of \theoref{sliderparallel} goes through word for word.  The rest of the section
gives the detailed definitions and describes the changes to the two key lemmas.

\paragraph{Color-looped-$(2,2)$ and color-looped-Laman graphs.}
A \emph{color-looped-(2,2) graph} is a looped graph with colored loops
that is looped-$(2,2)$, in addition, admits a coloring of its edges
into two forests so that each monochromatic tree spans exactly one loop of its color.

A \emph{color-looped-Laman graph} is a looped
graph with colored loops that is looped-Laman, and, in addition, admits a coloring
of its edges into red and blue forests so that each monochromatic tree spans
exactly one loop of its color.

\figref{color-looped-examples} shows examples.  The difference between these definitions and
the ones of looped-$(2,2)$ and looped-Laman graphs is that they are defined in terms of \emph{both}
graded sparsity counts \emph{and} a specific decomposition of the edges, depending on the colors of the
loops.
\begin{figure}[htbp]
\centering

\subfigure[]{\includegraphics[width=0.45\textwidth]{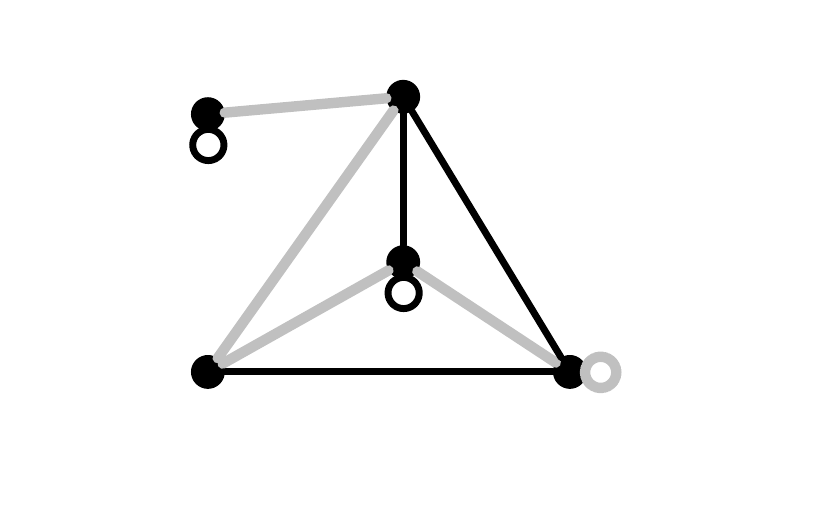}}
\subfigure[]{\includegraphics[width=0.45\textwidth]{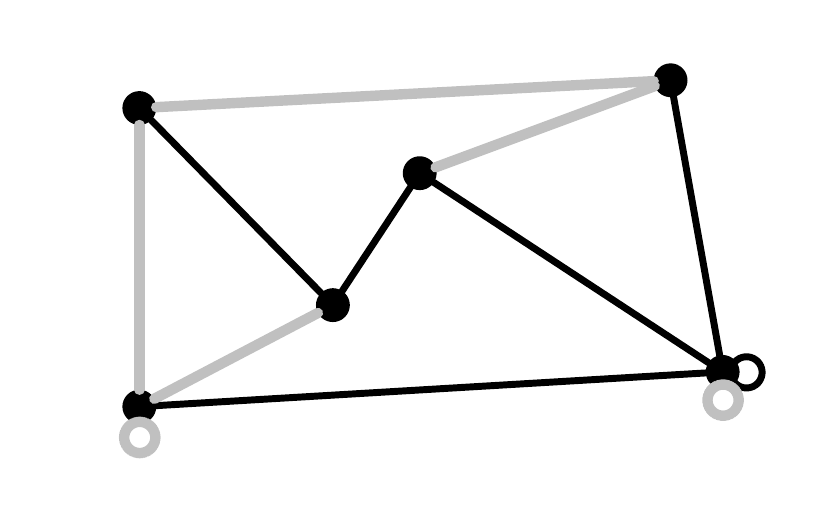}}
\caption{Examples of color-looped graphs, shown with forests certifying the color-looped property:
(a) a color-looped $(2,2)$-graph; (b) a color-looped Laman graph.  The colors red and blue
are represented by gray and black respectively.}
\label{fig:color-looped-examples}
\end{figure}

\paragraph{Realizing the $(2,0,2)$-graded-sparsity matroid for color-looped graphs.}
Recall that the matrix $\vec M_{2,0,2}(G)$  (see \figref{m22}(b)) realizing the $(2,0,2)$-sparsity
matroid has a row for each slider loop $i_j\in E$ with generic entries $c_{i_j}$ and $d_{i_j}$
in the two columns associated with vertex $i$.  For the color-looped case, we specialize
to the matrix $\vec M^{\mathsf{c}}_{2,0,2}(G)$, which has the same pattern as $\vec M_{2,0,2}(G)$, except:
\begin{itemize}
\item $c_{i_j}=1$ and $d_{i_j}=0$ for red loops $i_j\in E$
\item $c_{i_j}=0$ and $d_{i_j}=1$ for blue loops $i_j\in E$
\end{itemize}

The extension of the realization \lemref{202matrix}
to this case is the following.

\begin{lemma}\lemlab{202coloredmatrix}
Let $G$ be a color-looped graph on $n$ vertices with $m+c=2n$.
The matrix $\vec M^{\mathsf{c}}_{2,0,2}(G)$ has generic rank $2n$ if and only if $G$ is color-looped-$(2,2)$.
\end{lemma}
\begin{proof}
Modify the proof of \lemref{202matrix} to consider only decompositions into looped
forests in which each loop is assigned its correct color.  The definition of color-looped-$(2,2)$
graphs implies that one exists if and only if $G$ is color-looped-$(2,2)$.  As in the
uncolored case, the determinant is generically non-zero exactly when the required decomposition exists.
\end{proof}

\paragraph{Genericity for axis-parallel sliders.}
In the axis-parallel setting, our genericity condition is the following.
\begin{lemma}\lemlab{xyslidergenericity}
Let $G$ be a color-looped-Laman subgraph.  The set of directions and slider lines such that:
\begin{itemize}
\item The system $\vec S(G,\vec d,\vec n, \vec s)$ has rank $2n$ (and thus has a unique solution)
\item For all edges $ij\in E$, the system $\vec S(G/ij,\vec d,\vec n, \vec s)$ has no solution
\end{itemize}
is open and dense in $\mathbb{R}^{2m+3c}$.
\end{lemma}
\begin{proof}
Similar to the proof of \lemref{slidergenericity}, except  using \lemref{202coloredmatrix}.
\end{proof}

\section{Generic rigidity via direction network realization}\seclab{rigidity}
Having proven our main results on direction and direction-slider network realization,
we change our focus to the rigidity theory of bar-joint and bar-slider frameworks.

\subsection{Bar-joint rigidity}

In this section, we prove the Maxwell-Laman Theorem, following Whiteley \cite{Whiteley:1989p992}:
\laman

The difficult step of the proof is to show that a generic bar-joint framework $G(\vec p)$ with $m=2n-3$ edges is \emph{infinitesimally rigid}, that is the generic rank of the rigidity matrix $\vec M_{2,3}(G)$, shown in
\figref{rigidity-matrices}(a) has rank $2n-3$ if and only if $G$ is a Laman graph.
We will deduce this as a consequence of \theoref{parallel}.

\begin{figure}[htbp]
\centering

\subfigure[]{\includegraphics[width=0.48\textwidth]{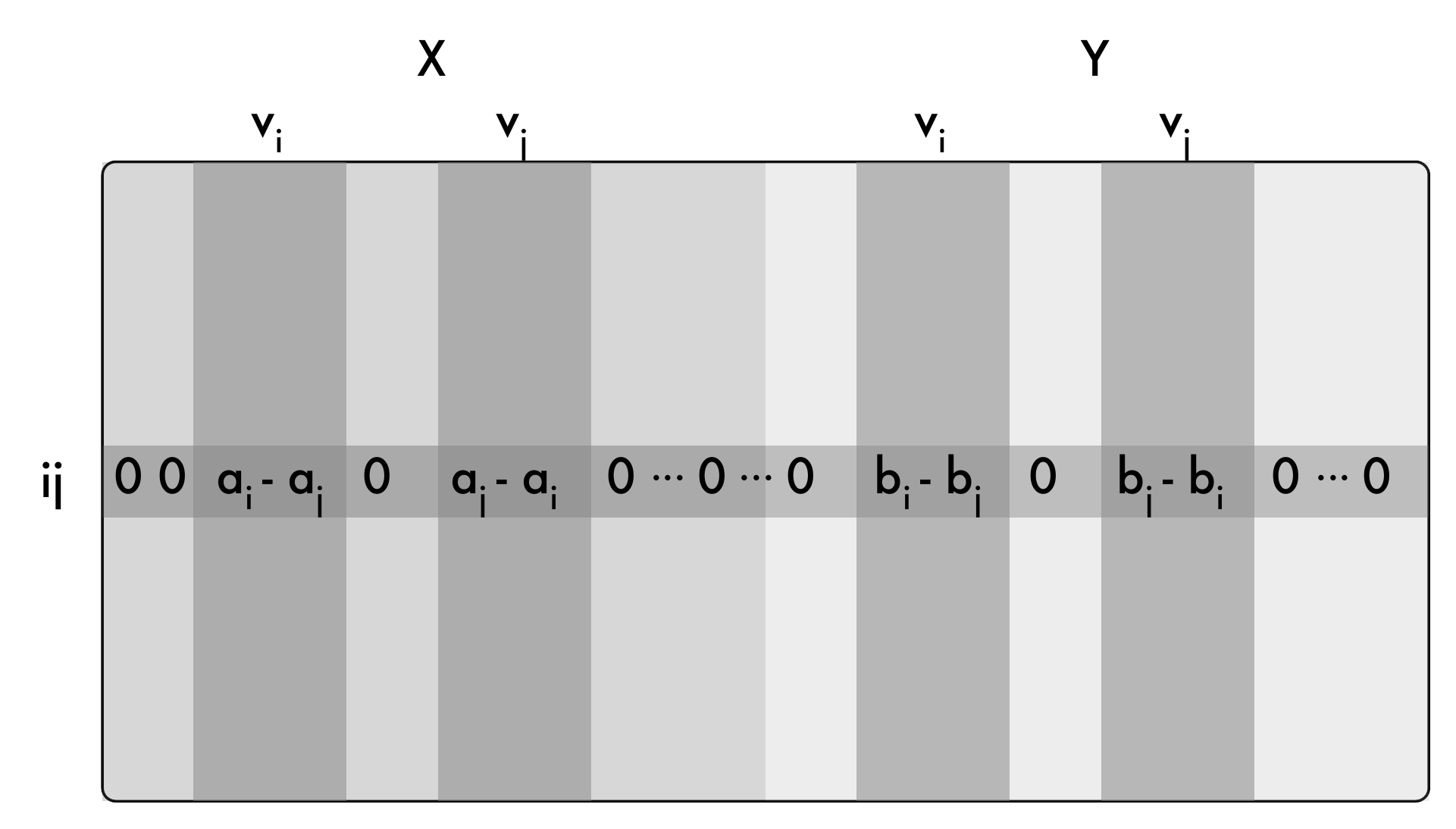}}\vspace{-0.2 in}
\subfigure[]{\includegraphics[width=0.48\textwidth]{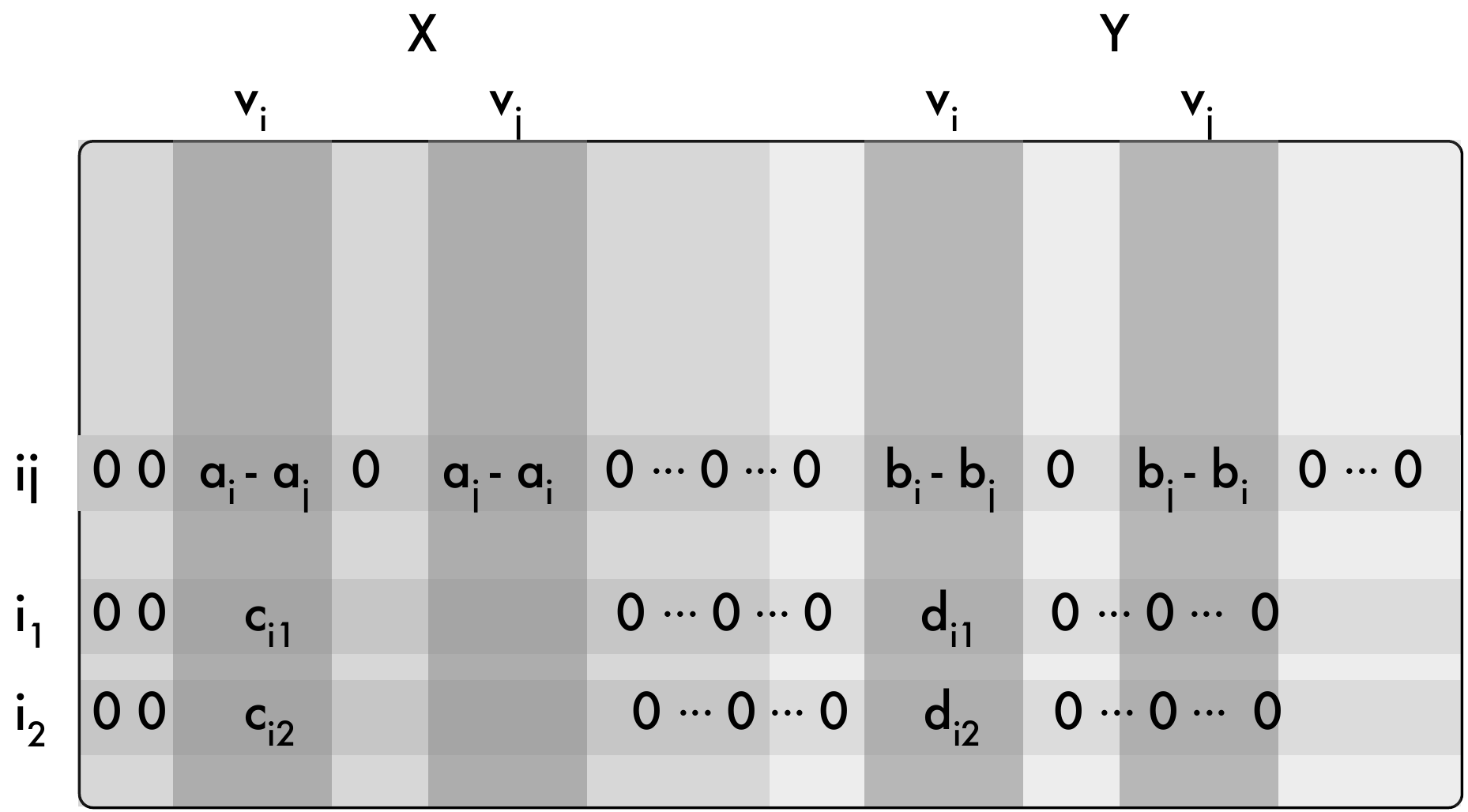}}
\caption{The pattern of the rigidity matrices: (a) the matrix $\vec M_{2,3}(G)$ for
bar-joint rigidity; (b) the matrix $\vec M_{2,0,3}(G)$ for bar-slider framework.}
\label{fig:rigidity-matrices}
\end{figure}

\paragraph{Proof of the Maxwell-Laman Theorem \ref{theo:laman}.}
\begin{proof}[Proof of the Maxwell-Laman Theorem \ref{theo:laman}]
Let $G$ be a Laman graph.  We need to show that the rank of the rigidity matrix
$\vec M_{2,3}(G)$ is $2n-3$ for a generic framework $G(\vec p)$.  We will do this
by constructing a point set $\hat{\vec p}$
for which the rigidity matrix has full rank.

Define a generic direction network $(G,\vec d)$ with its underlying graph $G$.  Because $\vec d$
is generic, the rank of $\vec M_{2,2}(G)$ is $2n-3$ for these directions $d_{ij}$, by \lemref{22matrix}.

By \theoref{parallel},
there is a point set $\hat{\vec p}$ such that $\hat{\vec p}_i\neq \hat{\vec p}_j$ for all $ij\in E$ and
$\hat{\vec p}_i-\hat{\vec p}_j = \alpha_{ij} \vec d_{ij}$ for some non-zero real number $\alpha_{ij}$.  Replacing $a_{ij}$ by $\alpha_{ij}(a_i-a_j)$
and $b_{ij}$ by $\alpha_{ij}(b_i-b_j)$ in $\vec M_{2,2}(G)$ and scaling
each row $1/\alpha_{ij}$ we obtain the rigidity matrix $\vec M_{2,3}(G)$.  It follows that
$\vec M_{2,3}(G)$ has rank $2n-3$ as desired.
\end{proof}

\paragraph{Remarks on Tay's proof of the Maxwell-Laman Theorem \cite{Tay93}.}  In \cite{Tay93}, Tay gives a
proof of the Maxwell-Laman Theorem based on so-called \emph{proper $3\mathsf{T}2$ decompositions} of Laman graphs (see \cite{colors} for a
detailed discussion).  The key idea is to work with what Tay calls a ``generalized framework'' that may have collapsed edges;
in the generalized rigidity matrix Tay defines, collapsed edges are simply assigned directions.  Tay then starts with a generalized framework
in which all edges are collapsed for which it is easy to prove the generalized rigidity matrix has full rank and then uses a
$3\mathsf{T}2$ decomposition to explicitly
perturb the vertices so that the rank of the generalized rigidity matrix is maintained as the endpoints
of collapsed edges are pulled apart.  At the end of the process, the generalized rigidity matrix coincides with the
Laman rigidity matrix.

In light of our genericity \lemref{parallelgeneric}, we can simplify Tay's approach.  Let $G$ be a Laman graph,
$\mathcal{D}\subset \mathbb{R}^{2m}$ is the set of directions for which $\vec M_{2,2}(G)$ has full rank,
and $\mathcal{P}\subset \mathcal{D}$ as
\(
\mathcal{P} = \{ \vec d\in \mathcal{D} : \exists \vec p\in \mathbb{R}^{2n}\, \forall ij\in E\, \vec d_{ij}=\vec p_i-\vec p_j\}
\);
i.e., $\mathcal{P}$ is the subset of $\mathcal{D}$ arising from the difference set of some planar point set.  From the
definition of $\mathcal{P}$ and arguments above, if $\vec d\in \mathcal{P}$ any realization of $(G,\vec d)$ interpreted
as a framework will be infinitesimally rigid.

\lemref{parallelgeneric} says that $\mathcal{P}$ is dense in $\mathcal{D}$ (and indeed $\mathbb{R}^{2m}$)
if and only if $G$ is a Laman graph.  In the language
of Tay's generalized frameworks, then, \lemref{parallelgeneric} gives a short,
existential proof that a full rank generalized framework can be perturbed into an infinitesimally rigid framework without
direct reference to \theoref{parallel}.  By making the connection to \theoref{parallel} explicit, we obtain
a canonical infinitesimally rigid realization that can be found using only linear algebra.

\subsection{Slider-pinning rigidity}
In this section we develop the theory of slider pinning rigidity and prove a Laman-type theorem for it.
\slider

We begin with the formal definition of the problem.
\paragraph{The slider-pinning problem.}
An abstract \emph{bar-slider framework} is a triple $(G,\bm{\ell},\bm{s})$ where $G=(V,E)$ is a
graph with $n$ vertices, $m$ edges and $c$ self-loops.  The vector $\bm{\ell}$ is a vector
of $m$ positive squared edge-lengths, which we index by the edges $E$ of $G$.  The vector $\bm{s}$ specifies a line
in the Euclidean plane for each self-loop in $G$, which we index as $i_j$ for the $j$th loop at vertex $i$;
lines are given by a normal vector $\vec n_{i_j}=(c_{i_j},d_{i_j})$ and a constant $e_{i_j}$.

A \emph{realization} $G(\vec p)$
is a mapping of the vertices of $G$ onto a point set $\vec p\in \left(\mathbb{R}^2\right)^n$
such that:
\begin{eqnarray}
||\vec p_i-\vec p_j||^2=\bm{\ell}_{ij} & \text{for all edges $ij\in E$} \eqlab{sledges}\\
\iprod{\vec p_i}{\vec n_{i_j}}=e_{i_j} & \text{for all self-loops $i_j\in E$}\eqlab{loops}
\end{eqnarray}
In other words, $\vec p$ respects all the edge lengths and assigns every
self-loop on a vertex to a point on the line specified by the corresponding slider.

\paragraph{Continuous slider-pinning.} \seclab{continuous}
The \emph{configuration space } $\mathcal{C}(G)\subset \left(\mathbb{R}^2\right)^n$ of a
bar-slider framework is defined as the space of real solutions to equations \eqref{sledges} and \eqref{loops}:
\[
\mathcal{C}(G) = \{\vec p\in \left(\mathbb{R}^2\right)^n : \text{$G(\vec p)$ is a realization of
$(G,\bm{\ell},\bm{s})$} \}
\]

A bar-slider framework $G(\vec p)$ is \emph{slider-pinning rigid} (shortly, \emph{pinned}) if $\vec p$ is an
isolated point in the configuration space $\mathcal{C}(G)$ and \emph{flexible} otherwise.
It is minimally pinned if it is pinned but fails to remain so if any edge or loop is removed.

\paragraph{Infinitesimal slider-pinning.}
Pinning-rigidity is a difficult condition to
establish algorithmically,
so we consider instead the following linearization of the problem.
Let $G(\vec p)$ be an axis-parallel bar-slider framework with $m$ edges and $c$
sliders.  The \emph{pinned rigidity matrix} (shortly rigidity matrix)
$\vec M_{2,0,3}(G(\vec p))$ is an $(m+c)\times 2n$ matrix that has one row for each edge
$ij\in E$ and self-loop $i_j\in E$, and one column for each vertex of $G$.  The columns
are indexed by the coordinate and the vertex, and we think of them as arranged into two blocks of $n$,
one for each coordinate.  The rows corresponding to edges have entries $a_i-a_j$ and $b_i-b_j$ for the
$x$- and $y$-coordinate columns of vertex $i$, respectively.  The $x$- and $y$-coordinate columns associated with
vertex $j$ contain the entries $a_j-a_i$ and $b_j-b_i$; all other entries are zero.  The row for a loop $i_j$ contains
entries $c_{i_j}$ and $d_{i_j}$ in the $x$- and $y$-coordinate columns for vertex $i$; all other entries are zero.
\figref{rigidity-matrices}(b) shows the pattern.

If $\vec M(G(\vec p))$ has rank $2n$ (the maximum possible), we say that $G(\vec p)$ is \emph{infinitesimally slider-pinning rigid} (shortly infinitesimally pinned); otherwise it is \emph{infinitesimally flexible}.  If $G(\vec p)$
is infinitesimally pinned but fails to be so after removing any edge or loop from $G$, then it is \emph{minimally infinitesimally pinned}.

The pinned rigidity matrix arises as the differential of the system given by \eqref{edges} and \eqref{loops}.
Its rows span the normal space of $\mathcal{C}$ at $\vec p$ and the kernel is the
tangent space $T_{\vec p} \mathcal{C}(G)$ at $\vec p$.  With this observation, we can show that
infinitesimal pinning implies pinning.

\begin{lemma}\lemlab{infinitesimal-pinning}
Let $G(\vec p)$ be a bar-slider framework.  If $G(\vec p)$ is infinitesimally pinned, then $G(\vec p)$
is pinned.
\end{lemma}
In the proof, we will need the \emph{complex configuration space} $\mathcal{C}_{\mathbb{C}}(G)$ of
$G$, which is the solution space to the system \eqref{edges} and \eqref{loops} in $\left(\mathbb{C}^2\right)^n$.
The rigidity matrix has the same form in this setting.
\begin{proof}
Since $\vec M(G(\vec p))$ has $2n$ columns, if its rank is $2n$, then its kernel is the just the zero vector.
By the observation above, this implies that the tangent space $T_{\vec p} \mathcal{C}_{\mathbb{C}}(G)$ is
zero-dimensional.  A fundamental result result of algebraic geometry \cite[p. 479, Theorem 8]{cox:little:oshea:iva:1997}
says that the irreducible components
of $\mathcal{C}_{\mathbb{C}}(G)$ through $\vec p$ have dimension bounded by the dimension of the tangent space at
$\vec p$.

It follows that $\vec p$ is an isolated point
in the complex configuration space and, by inclusion, in the real configuration
space.
\end{proof}

\subsection{Generic bar-slider frameworks}\seclab{genericity}
Although \lemref{infinitesimal-pinning} shows that infinitesimal pinning implies
pinning, the converse is not, in general, true.  For example, a bar-slider framework that is combinatorially
a triangle with one loop on each vertex is pinned, but not infinitesimally pinned, in a realization where
the sliders are tangent to the circumcircle.

For \emph{generic}
bar-slider frameworks, however, pinning and infinitesimal pinning coincide.
A realization $G(\vec p)$ bar-slider framework  is generic if the rigidity matrix attains its
maximum rank at $\vec p$; i.e.,
$\operatorname{rank}\left(\vec M(\vec p)\right)\ge \operatorname{rank}\left(\vec M(\vec p)\right)$
for all $\vec q\in \mathbb{R}^{2n}$.

We reformulate genericity in terms of the \emph{generic pinned rigidity matrix}
$\vec M(G)$, which is defined to have the same pattern as the pinned rigidity matrix, but with
entries that are formal polynomials in variables $a_i$, $b_i$, $c_{i_j}$, and $d_{i_j}$.
The rank of the generic rigidity matrix is defined as the largest integer $r$ for which there is
an $r\times r$ minor of $\vec M(G)$ which is not \emph{identically zero} as a formal polynomial.

A \emph{graph} $G$ is defined to be \emph{generically infinitesimally rigid}
if its generic rigidity matrix
$\vec M(G)$ has rank $2n$ (the maximum possible).

\subsection{Proof of \theoref{slider}}
We are now ready to give the proof of our Laman-type \theoref{slider} for bar-slider frameworks.
\slider

\begin{proof}
Let $G$ be looped-Laman.  We will construct a point set $\hat{\vec p}$, such that the
bar-slider framework $G(\hat{\vec p})$ is infinitesimally pinned.

Fix a generic direction-slider network $(G,\vec d,\vec n,\vec s)$ with underlying graph $G$.
By \lemref{202matrix}, $\vec M_{2,0,2}(G)$ has rank $2n$.
Applying \theoref{sliderparallel}, we obtain a point set $\hat{\vec p}$ with $\hat{\vec p}_i\neq \hat{\vec p}_j$
for all edges $ij\in E$ and $\vec p_i-\vec p_j=\alpha_{ij}\vec d_{ij}$.  Substituting in to $\vec M_{2,0,2}(G)$ and
rescaling shows the rank of $\vec M_{2,0,3}(G)$ is $2n$.
\end{proof}

\end{document}